\documentclass[11pt]{amsart}

\issueinfo{00}{0}{Month}{Year}
\copyrightinfo{2008}{University of Calgary}
\pagespan{1}{2}
\PII{ISSN 1715-0868}

\usepackage{graphicx}

\usepackage{amssymb}
\usepackage{hyperref}
\usepackage[capitalize]{cleveref}

\renewcommand{\MR}[1]{\href{http://www.ams.org/mathscinet-getitem?mr=#1}{MR#1}}

\numberwithin{equation}{section}

\theoremstyle{plain}
\newtheorem{thm}[equation]{Theorem}
\newtheorem{prop}[equation]{Proposition}
\newtheorem{cor}[equation]{Corollary}
\newtheorem{lem}[equation]{Lemma}

\newtheorem{problem}[equation]{Problem}

\crefmultiformat{thm}{Theorems~#2#1#3}{ and~#2#1#3}{, #2#1#3}{, and~#2#1#3}

\Crefname{cor}{Corollary}{Corollaries}
\Crefname{lem}{Lemma}{Lemmas}
\Crefname{prop}{Proposition}{Propositions}
\Crefname{thm}{Theorem}{Theorems}

\theoremstyle{definition}
\newtheorem{notation}[equation]{Notation}
\newtheorem{defn}[equation]{Definition}

\theoremstyle{definition}
\newtheorem{rem}[equation]{Remark}

\newtheorem{warn}[equation]{Warning}

\crefformat{rems}{Remark~#2#1#3}

\Crefname{figure}{Figure}{Figures}
\crefmultiformat{figure}{Figures~#2#1#3}{ and~#2#1#3}{, #2#1#3}{, and~#2#1#3}

\Crefname{page}{page}{pages}

 \newcounter{case}
 \newenvironment{case}[1][\unskip]{\refstepcounter{case}\bf
 \medskip \noindent Case \thecase\ #1.\normalfont\em\ }{\unskip\upshape}
 \renewcommand{\thecase}{\arabic{case}}
 \crefformat{case}{Case~#2#1#3}

 \newcounter{subcase}
 \newenvironment{subcase}[1][\unskip]{\refstepcounter{subcase}\bf
 \medskip \noindent \hskip2\parindent Subcase \thesubcase\ #1. \it}{\unskip\upshape}
\numberwithin{subcase}{case}
\crefformat{subcase}{Subcase~#2#1#3}
\crefname{subcase}{subcase}{subcases}
\Crefname{subcase}{Subcase}{Subcases}

 \renewcommand{\Cref}{\cref}
 \newcommand{\pref}[1]{(\ref{#1})}
 \newcommand{\fullref}[2]{\ref{#1}\pref{#1-#2}}
 \newcommand{\fullcref}[2]{\cref{#1}\pref{#1-#2}}
 
 \newcommand{\fullCref}[2]{\Cref{#1}\pref{#1-#2}}
\newcommand{\csee}[1]{(see \cref{#1})}
\newcommand{\fullcsee}[2]{(see \fullcref{#1}{#2})}

\newcommand{\boardonly}{\mathcal{B}}
\newcommand{\board}[2]{\boardonly_{#1,#2}}
\newcommand{\Bmn}{\board{m}{n}}
\newcommand{\subdiag}{\mathord{S}}
\newcommand{\hampath}{\mathcal{H}}
\newcommand{\init}{\iota}
\newcommand{\term}{\tau}

\newcommand{\mf}{\mathord{\mathchoice
	{\hbox to 0pt{\vrule width 9.5pt height 2.75pt depth -1.95pt\hss}m}%
	{\hbox to 0pt{\vrule width 9.5pt height 2.75pt depth -1.95pt\hss}m}%
	{\hbox to 0pt{\vrule width 7.5pt height 2pt depth -1.6pt\hss}m}%
	{\hbox to 0pt{\vrule width 6.5pt height 1.6pt depth -1.1pt\hss}m}}}
\newcommand{\mm}{\mathord{\mf^-}}
\renewcommand{\mp}{\mathord{\mf^+}}
\newcommand{\nf}{\mathord{\mathchoice
	{\hbox to 0pt{\vrule width 6.5pt height 2.75pt depth -1.95pt\hss}n}%
	{\hbox to 0pt{\vrule width 6.5pt height 2.75pt depth -1.95pt\hss}n}%
	{\hbox to 0pt{\vrule width 5pt height 2pt depth -1.6pt\hss}n}%
	{\hbox to 0pt{\vrule width 4.3pt height 1.6pt depth -1.1pt\hss}n}}}

\newcommand{\nm}{\mathord{\nf^-}}
\newcommand{\np}{\mathord{\nf^+}}

\newcommand{\NN}{\mathbb{N}}
\newcommand{\ZZ}{\mathbb{Z}}

\makeatletter
\newcommand{\noprelistbreak}{\smallskip\@nobreaktrue\nopagebreak} 
\makeatother

\usepackage{color}
\newcommand{\refnote}[1]{%
	\marginpar{%
		\color{blue}
		\vbox to 0pt{\vss
		$\begin{pmatrix} \hbox{note} \\ \hbox{\ref{#1}} \end{pmatrix}$%
		\vskip -11pt}}}
\newcommand{\refnotelower}[1]{%
	\marginpar{%
		\color{blue}
		\vbox to 0pt{\vskip 2pt
		$\begin{pmatrix} \hbox{note} \\ \hbox{\ref{#1}} \end{pmatrix}$%
		\vss}}}
\theoremstyle{definition}
\newtheorem{aid}{}
\newcommand{\oldendaid}{}
\let\oldendaid=\endaid
\renewcommand{\endaid}{\oldendaid\bigskip\hrule width\textwidth \bigbreak}
\numberwithin{aid}{section}

\begin{document}

\title[Hamiltonian paths in $m \times n$ projective checkerboards]%
{Hamiltonian paths in $m \times n$ \\ projective checkerboards}

\author{Dallan McCarthy and Dave Witte Morris}
\address{Department of Mathematics and Computer Science,
University of Lethbridge, Lethbridge, Alberta, T1K~6R4, Canada}
\email{mccarthyd@uleth.ca}
{\catcode`~ = 12 
\email{Dave.Morris@uleth.ca, http://people.uleth.ca/~dave.morris/}
}

\date{\today.} 

\begin{abstract}
For any two squares $\init$ and~$\term$ of an $m \times n$ checkerboard, we determine whether it is possible to move a checker through a route that starts at~$\init$, ends at~$\term$, and visits each square of the board exactly once. Each step of the route moves to an adjacent square, either to the east or to the north, and may step off the edge of the board in a manner corresponding to the usual construction of a projective plane by applying a twist when gluing opposite sides of a rectangle. 
This generalizes work of M.\,H.\,Forbush et al.\ for the special case where $m = n$.
\end{abstract}

{
\mathversion{bold} 
\maketitle
}

\section{Introduction}

Place a checker in any square of an $m \times n$ checkerboard (or chessboard). We determine whether it is possible for the checker to move through the board, visiting each square exactly once. (In graph-theoretic terminology, we determine whether there is a hamiltonian path that starts at the given square.) Although other rules are also of interest (such as the well-known knight moves discussed in \cite{WatkinsBook} and elsewhere), we require each step of the checker to move to an adjacent square that is either directly to the east or directly to the north, except that we allow the checker to step off the edge of the board. 

Torus-shaped checkerboards are already understood (see, for example, \cite{GallianWitte-HamCbds}), so we allow the checker to step off the edge of the board in a manner that corresponds to the usual procedure for creating a projective plane, by applying a twist when gluing each edge of a rectangle to the opposite edge:

\begin{defn}[cf.\ {\cite[Defn.~1.1]{Forbush}}]
 The squares of an $m \times n$ checkerboard can be naturally identified
with the set $\ZZ_m \times \ZZ_n$ of ordered pairs $(p,q)$ of integers with $0 \le p \le m-1$ and $0 \le q \le n-1$.
 Define $E \colon \ZZ_m \times \ZZ_n \to \ZZ_m \times \ZZ_n$ and $N \colon \ZZ_m \times \ZZ_n \to \ZZ_m \times \ZZ_n$ by
 \begin{align*}
 (p,q)E &= 
  \begin{cases}
 (p+1,q)& \text{if $p<m-1$}\\
 (0,n-1-q)& \text{if $p = m-1$}\\
 \end{cases}
 \\
 \intertext{and}
 (p,q)N &= 
 \begin{cases}
 (p,q+1)& \text{if $q<n-1$} \\
 (m-1-p,0)& \text{if $q = n-1$}
 . \end{cases}
 \end{align*}
 The  \emph{$m \times n$ projective checkerboard}~$\Bmn$ is the digraph whose vertex
set is $\ZZ_m \times \ZZ_n$, with a directed edge from~$\sigma$ to~$\sigma E$ and from~$\sigma$
to~$\sigma N$, for each $\sigma \in \ZZ_m \times \ZZ_n$. We usually refer to the vertices
of~$\Bmn$ as  \emph{squares}.
 \end{defn}

In a projective checkerboard~$\Bmn$ (with $m,n \ge 3$), only certain squares can be
the initial square of a hamiltonian path, and only certain squares can be the
terminal square. A precise determination of these squares was found by M.\,H.\,Forbush et al.\ \cite{Forbush} in the special
case where $m = n$ (that is, when the checkerboard is square, rather than
properly rectangular).
In this paper, we find both the initial squares and the terminal squares in the general case. (Illustrative examples appear in \cref{Initmx5Fig,Initmx10Fig,Termmx5Fig,Termmx10Fig} on \cpageref{Initmx5Fig,Initmx10Fig,Termmx5Fig,Termmx10Fig}.)

\begin{notation}
 For convenience, let 
 	\begin{align*}
 	\mf &= \lfloor m/2 \rfloor,
	&  
	\mm &= \lfloor (m-1)/2 \rfloor,
	&
	\mp &= \lfloor (m+1)/2 \rfloor = \lceil m/2 \rceil = \mm + 1, 
	\\
	\nf &= \lfloor n/2 \rfloor,
	&
	\nm &= \lfloor (n-1)/2 \rfloor, 
	& 
	\np &= \lfloor (n+1)/2 \rfloor \ = \lceil n/2 \rceil \ = \nm + 1
	. \end{align*}
 \end{notation} 
 
\begin{thm} \label{InitSquares}
Assume $m \ge n \ge 3$. There is a hamiltonian path in $\Bmn$ whose initial square is $(p,q)$ if and only if either:
\noprelistbreak
	\begin{enumerate} \itemsep=\smallskipamount 
	
	\item $p = 0$ and $\nm \le q \le n-1$,
	or
	
	\item $\nm \le p \le m-1$ and $q = 0$,
	or
	
	\item $\mp \le p \le m-1$ and $q = \nf$,
	or
	
	\item $0 \le p \le \nf$ and $q = \nm$,
	or
	
	\item $\mf \le p \le m - \np$ and $\nf + 1 \le q \le n-1$,
	or
	
	\item $\nm \le p \le \mm$ and $0 \le q \le \nf$.

	\end{enumerate}
\end{thm}

By rotating the checkerboard $180^\circ$ (cf.\ \cref{InverseHP}), this \lcnamecref{InitSquares} can be restated as follows:

\begin{thm} \label{TermSquares}
Assume $m \ge n \ge 3$. There is a hamiltonian path in $\Bmn$ whose terminal square is $(x,y)$ if and only if either:
\noprelistbreak
	\begin{enumerate} \itemsep=\smallskipamount 
	
	\item \label{TermSquares-m}
	$x = m-1$ and $0 \le y \le \nf$,
	or
	
	\item \label{TermSquares-n}
	$0 \le x \le m - \np$ and $y = n - 1$,
	or
	
	\item \label{TermSquares-nm}
	$0 \le x \le \mf - 1$ and $y = \nm$,
	or
	
	\item \label{TermSquares-nf}
	$m - \nf - 1 \le x \le m - 1$ and $y = \nf$,
	or
	
	\item \label{TermSquares-bottom}
	$\nm \le x \le \mm$ and $0 \le y \le \nm - 1$,
	or
	
	\item \label{TermSquares-top}
	$\mf \le x \le m - \np$ and $\nm \le y \le n - 1$.
	
	\end{enumerate}
\end{thm}

\begin{figure}[b]
\hrule\bigskip
\hbox to \textwidth{%
\includegraphics{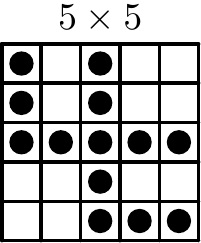}
\hfil
\includegraphics{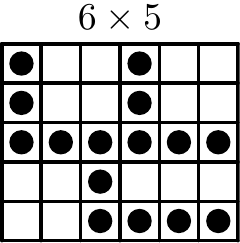}
\hfil
\includegraphics{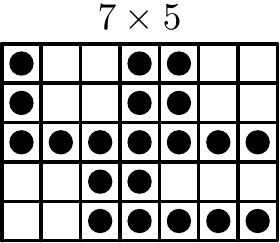}
\hfil
\includegraphics{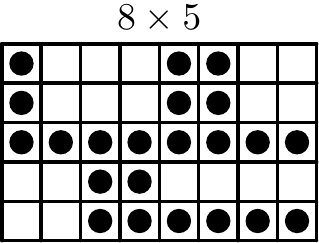}%
}
\bigskip
\hbox to \textwidth{%
\includegraphics{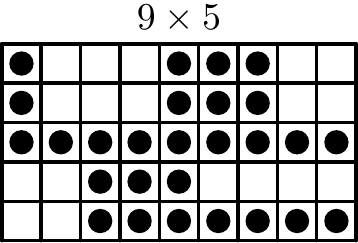}
\hfil
\includegraphics{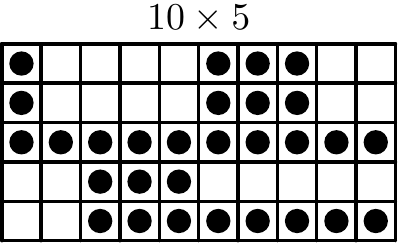}
\hfil
\includegraphics{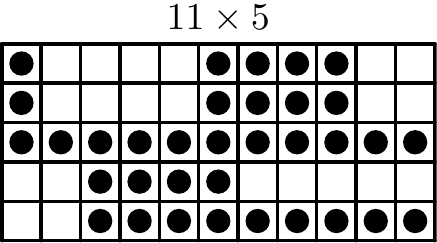}%
}
\bigskip
\centerline{
\includegraphics{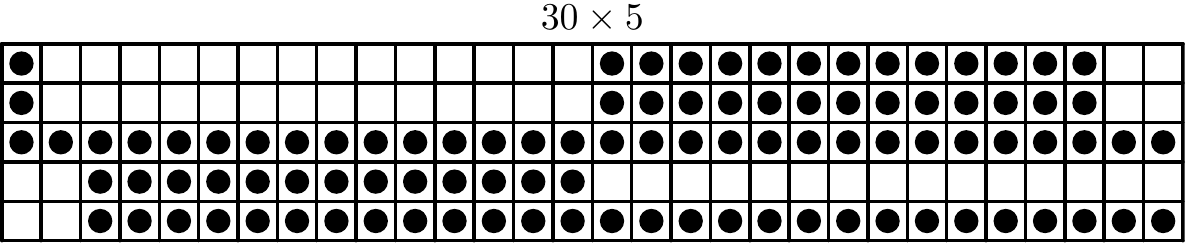}
}
\caption{The initial squares (\lower1pt\hbox{\larger[2]$\bullet$})
of hamiltonian paths in some $m \times 5$ projective checkerboards.}
\label{Initmx5Fig}
\end{figure}

\begin{figure}[b]
\hrule\bigskip
\centerline{
\includegraphics{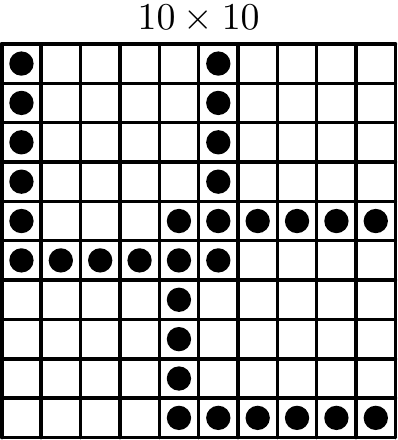}
\hfil
\includegraphics{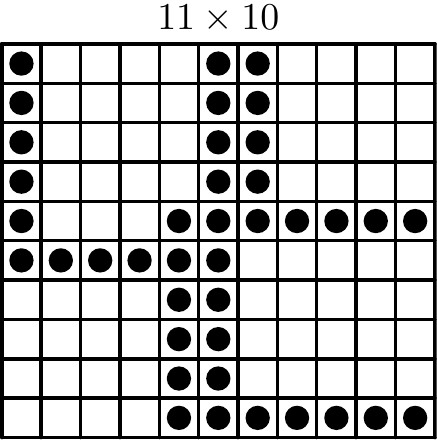}
}
\bigskip
\centerline{
\includegraphics{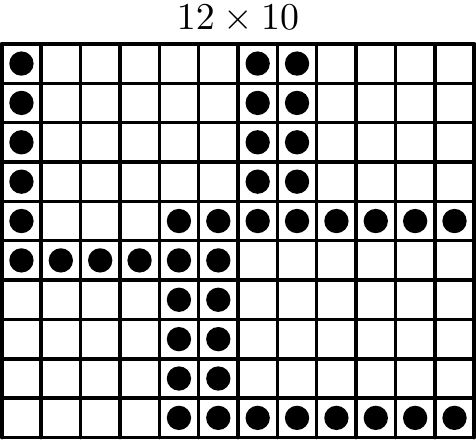}
\hfil
\includegraphics{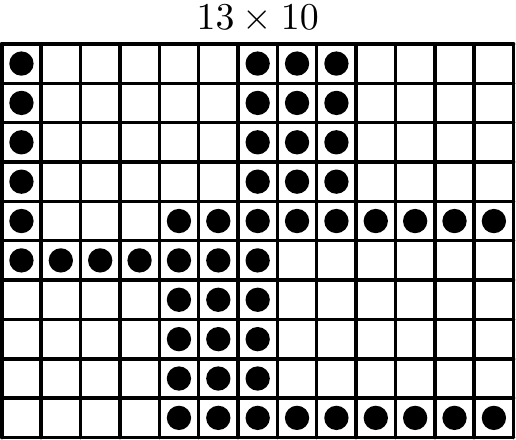}
}
\bigskip
\centerline{
\includegraphics{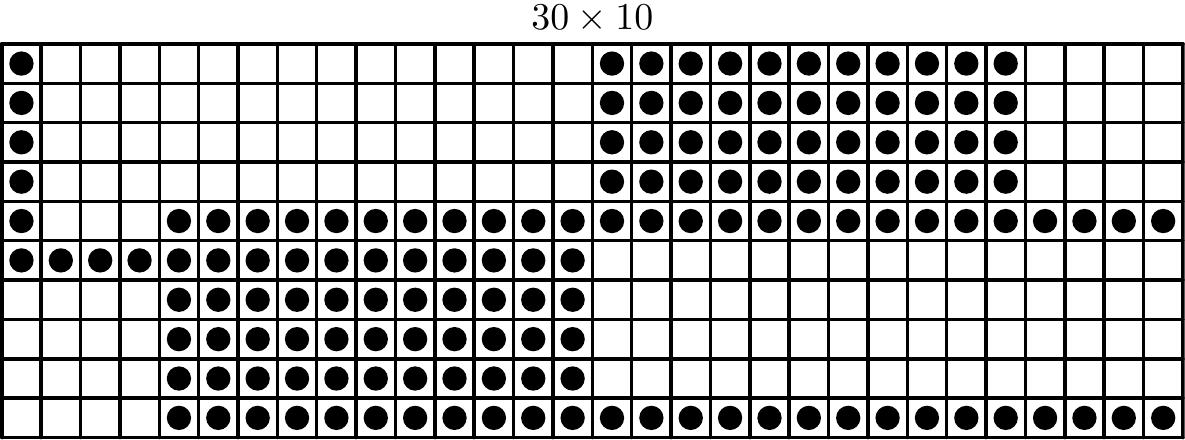}
}
\caption{The initial squares (\lower1pt\hbox{\larger[2]$\bullet$}) of hamiltonian paths in some $m \times 10$ projective checkerboards.}
\label{Initmx10Fig}
\end{figure}

\begin{figure}[b]
\hrule\bigskip
\hbox to \textwidth{%
\includegraphics{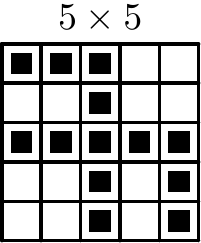}
\hfil
\includegraphics{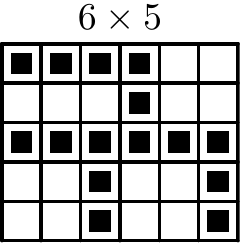}
\hfil
\includegraphics{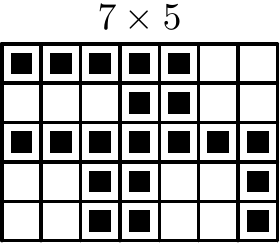}
\hfil
\includegraphics{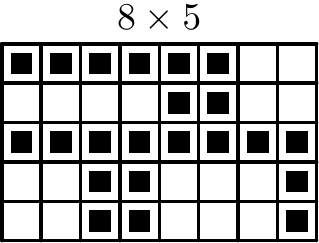}%
}
\bigskip
\hbox to \textwidth{%
\includegraphics{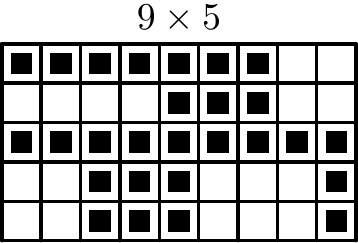}
\hfil
\includegraphics{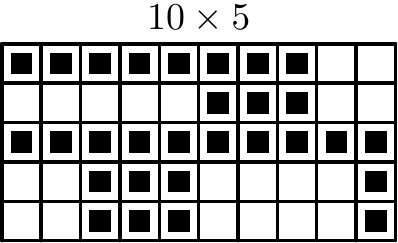}
\hfil
\includegraphics{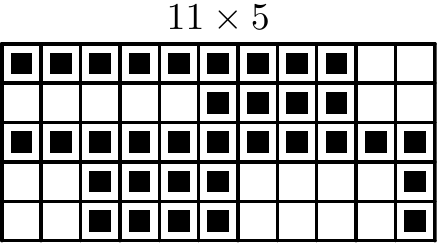}%
}
\bigskip
\centerline{
\includegraphics{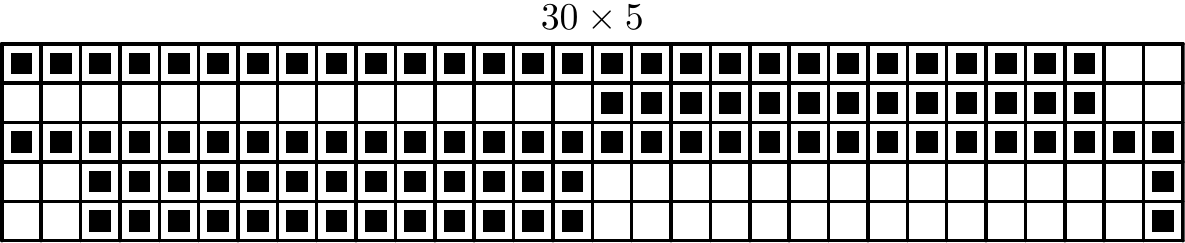}
}
\caption{The terminal squares (\vrule width 6pt height 6pt depth 0pt) of hamiltonian paths in some $m \times 5$ projective checkerboards.}
\label{Termmx5Fig}
\end{figure}

\begin{figure}[b]
\hrule\bigskip
\centerline{
\includegraphics{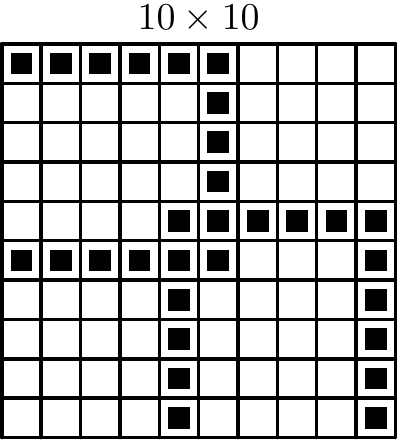}
\hfil
\includegraphics{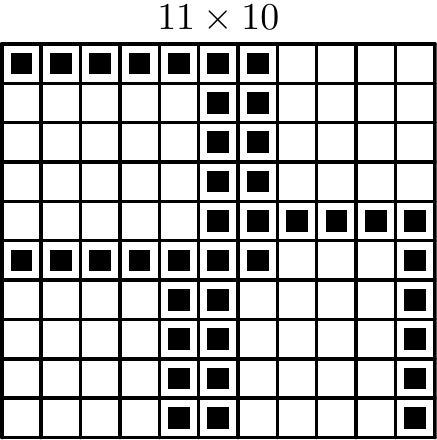}
}
\bigskip
\centerline{
\includegraphics{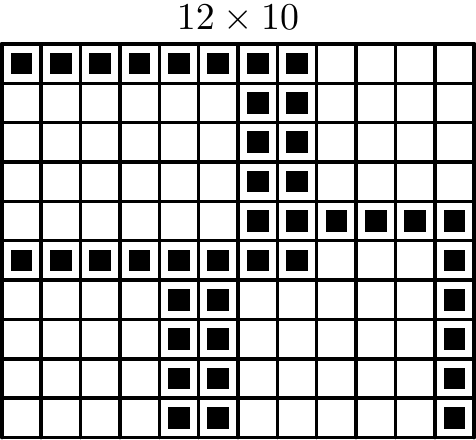}
\hfil
\includegraphics{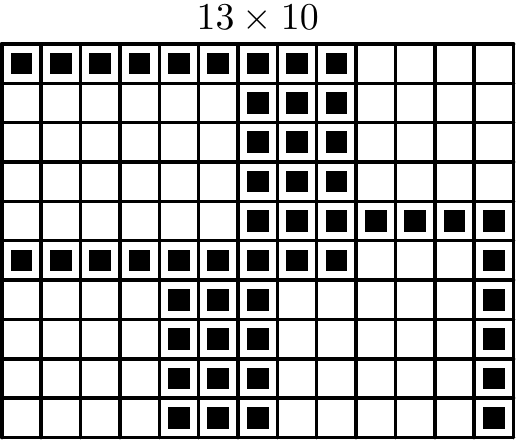}
}
\bigskip
\centerline{
\includegraphics{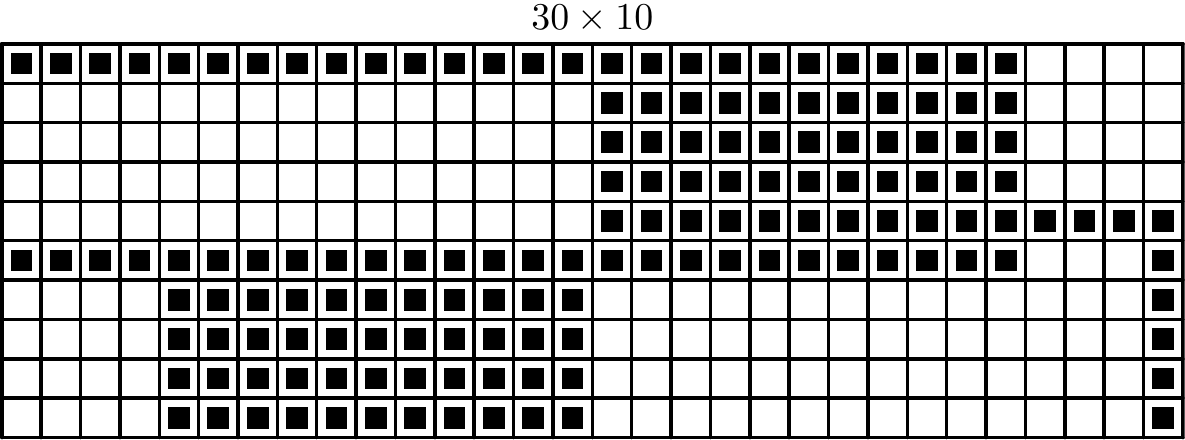}
}
\caption{The terminal squares (\vrule width 6pt height 6pt depth 0pt) of hamiltonian paths in some $m \times 10$ projective checkerboards.}
\label{Termmx10Fig}
\end{figure}

\begin{rem} \label{n=2HC}
By symmetry, there is no harm in assuming that $m \ge n$ when studying $\Bmn$. Furthermore, if $\min(m,n) \le 2$, then it is easy to see that $\Bmn$ has a hamiltonian cycle. Therefore, every square is the initial square of a hamiltonian path (and the terminal square of some other hamiltonian path) in the cases not covered by \cref{InitSquares,TermSquares}. 
\end{rem}

For any square $\init$ of~$\Bmn$, we determine not only whether there exists a hamiltonian path that starts at~$\init$, but also the terminal square of each of these hamiltonian paths. This more detailed result is stated and proved in \cref{GeneralSect}. It yields \cref{InitSquares,TermSquares} as corollaries. As preparation for the proof, we recall some known results in \cref{PrelimSect}, consider a very helpful special case in \cref{AllNorth}, and explain how to reduce the general problem to this special case in \cref{ReductionSect}.

\begin{rem}
Suppose $m$ and~$n$ are large. It follows from \cref{InitSquares} that a square in $\Bmn$ is much less likely to be the starting point of a hamiltonian path when the checkerboard is square than when it is very oblong:%
\noprelistbreak
	\begin{itemize}
	\item If $m = n$, then only a small fraction ($\approx 3/n$) of the squares are the initial square of a hamiltonian path. 
	\item In contrast, if $m$ is much larger than~$n$, then about half of the squares are the initial square of a hamiltonian path.
	\end{itemize}
\end{rem}

The following question remains open (even when $m = n$).

\begin{problem}
Which squares are the terminal square of a hamiltonian path that starts at $(0,0)$ in an $m \times n$ Klein-bottle checkerboard, where 
	$$ \text{$(m-1,q)E = (0, n - 1 - q)$ \ and \ $(p,n-1)N = (p, 0)$} .$$
\end{problem}

\section{Preliminaries: definitions, notation, and previous results} \label{PrelimSect}

We reproduce some of the elementary, foundational content of \cite{Forbush}, slightly modified to eliminate that paper's standing assumption that $m = n$. 

\begin{notation}[{}{\cite[Notation~3.1]{Forbush}}]
 We use $[\sigma ](X_1X_2\cdots X_k)$, where $X_i \in \{E,N\}$, to denote the
walk in~$\Bmn$ that visits (in order) the squares
 $$ \sigma ,~\sigma X_1,~ \sigma X_1X_2,~\dots,~ \sigma X_1X_2\dots X_k .$$
 \end{notation}

\begin{defn}[{}{\cite[Defn.~2.14]{Forbush}}]
 For $\sigma = (p,q) \in \Bmn$, we define the \emph{inverse} of~$\sigma$ to be
 	$ \widetilde\sigma = (m-1-p,n-1-q) $.
 \end{defn}
 
 \begin{rem}
 $\widetilde\sigma$ can be obtained from~$\sigma$ by rotating the checkerboard 180 degrees.
 \end{rem}
 
 \begin{prop}[{}{\cite[Prop.~2.15]{Forbush}}] \label{InverseHP}
 If there is a hamiltonian path from $\init$ to~$\term$ in~$\Bmn$, then there is also a hamiltonian path from $\widetilde\term$ to~$\widetilde\init$.
 
 More precisely, if $\hampath = [\init](X_1 X_2 \cdots X_k)$ is a hamiltonian path from $\init$ to~$\term$, then the \emph{inverse} of~$\hampath$ is the hamiltonian path
 	$\widetilde\hampath = [\widetilde\term]( X_k X_{k-1} \cdots X_1 )$ from $\widetilde\term$ to~$\widetilde\init$.
 \end{prop}
 
 \begin{defn}[{}{\cite[Defn.~2.14 and Prop.~2.15]{Forbush}}]
 If $m = n$, then:
 	\begin{itemize}
	\item The \emph{transpose}~$\sigma^*$ of any square~$\sigma$ of~$\Bmn$ is defined by $(p,q)^* = (q,p)$.
 
	\item For a hamiltonian path $\hampath = [\init](X_1 X_2 \cdots X_k)$ from $\init$ to~$\term$, the \emph{transpose} of~$\hampath$ is the hamiltonian path
 	$\hampath^* = [\init^*]( X_1^* X_2^* \cdots X_k^* )$ from $\init^*$ to~$\term^*$,
	where $E^* = N$ and $N^* = E$.
	\end{itemize}
 \end{defn}

\subsection{Direction-forcing diagonals} \label{diag-sect}

\begin{defn}[{}{\cite[Defn.~2.1]{Forbush}}] Define a symmetric, reflexive relation~$\sim$ on the set of
squares of~$\Bmn$ by $\sigma \sim \term$ if
 $$ \{\sigma E,\sigma N\} \cap \{\term E , \term N\} \not= \emptyset .$$
 The equivalence classes of the transitive closure of~$\sim$ are
\emph{direction-forcing diagonals}. For short, we refer to them simply as
\emph{diagonals}. Thus, the diagonal containing~$\sigma$ is
 $$\{\sigma ,\sigma NE^{-1} ,\sigma (NE^{-1})^2 , \dots, \sigma EN^{-1} \}.$$
 \end{defn}

\begin{notation}[{}{\cite[Notn.~2.3]{Forbush}}]
 For $0 \le i \le m + n - 2$, let 
 	$$\subdiag_i = \{\, (p,q) \in \Bmn \mid p + q = i \,\} .$$
We call $S_i$ a \emph{subdiagonal}.
 \end{notation}

\begin{prop}[{}{\cite[Prop.~2.4]{Forbush}}]  \label{diags(m+n-3)}
 For each~$i$ with $0 \le i \le m+n-3$, the set $D_i = \subdiag_i \cup \subdiag_{m+n-3-i}$
is a diagonal. The only other diagonal $D_{m+n-2}$ consists of the single square
$(m-1,n-1)$.
 \end {prop} 

\begin{cor}[{}{\cite[Notn.~2.5]{Forbush}}]
 Let $D$ be a diagonal, other than $D_{m+n-2}$. Then we may write $D = \subdiag_a \cup \subdiag_b$ with
$a \le b$ and $a + b = m + n - 3$.
 \end{cor}

\begin{defn}[{}{\cite[Defn.~2.7]{Forbush}}]
If $\hampath$ is a hamiltonian path in~$\Bmn$, then the diagonal containing the
terminal square~$\term$ is called the \emph{terminal} diagonal of~$\hampath$. All
other diagonals are \emph{non-terminal} diagonals.
 \end{defn}

\begin{defn}[{}{\cite[Defn.~2.8]{Forbush}}]
 Let $\hampath$ be a hamiltonian in~$\Bmn$. A square~$\sigma$ \emph{travels
east} (in~$\hampath$) if the edge from~$\sigma$ to~$\sigma E$ is in~$\hampath$. Similarly,
$\sigma$ \emph{travels north} (in~$\hampath$) if the edge from~$\sigma$ to~$\sigma N$
is in~$\hampath$.
 \end{defn}

The following important observation is essentially due to R.\,A.\,Rankin \cite[proof of Thm.~2]{Rankin-CampGrpThy}.

\begin{prop}[{}{\cite[Prop.~2.9]{Forbush}, cf.~\cite[Prop.\ on p.~82]{Housman}, \cite[Lem.~6.4c]{Curran}}] \label{DiagMustTravel}
 If $\hampath$ is a hamiltonian path in~$\Bmn$, then, for each non-terminal
diagonal~$D$, either every square in~$D$ travels north, or every square in~$D$
travels east. For short, we say that either $D$ travels north or $D$ travels
east.
 \qed
 \end{prop}

\begin{prop}[{}{\cite[Prop.~2.10]{Forbush}, cf.~\cite[Lem.~6.4b]{Curran}}] \label{term-travels}
 Let $D$ be the terminal diagonal of a hamiltonian path~$\hampath$ in~$\Bmn$,
with initial square~$\init$ and terminal square~$\term$, and let $\sigma \in D$.
 \begin{itemize}
 \item if $\term N \not= \init$, then $\term NE^{-1}$ travels east;
 \item if $\term E \not= \init$, then $\term EN^{-1}$ travels north;
 \item if $\sigma$ travels east and $\sigma N\not=\init$, then $\sigma NE^{-1}$
travels east; 
 \item if $\sigma$ travels east, then $\sigma E N^{-1}$ does not travel north;
 \item if $\sigma$ travels north and $\sigma E\not=\init$, then $\sigma
EN^{-1}$ travels north; and
 \item if $\sigma$ travels north, then $\sigma N E^{-1}$ does not travel east.
  \qed
 \end{itemize}
 \end{prop}

\begin{cor}[{}{\cite[Cor.~2.11]{Forbush}, cf.~\cite[Lem.~6.4a]{Curran}}] \label{iotaE=tau}
 If $\hampath$ is a hamiltonian path in~$\Bmn$, then the diagonal that contains
$\init E^{-1}$ and $\init N^{-1}$ is the terminal diagonal.
 \end{cor}

The following \lcnamecref{circ-orient} follows from \cref{term-travels} by
induction. 

\begin{cor}[{}{\cite[Cor.~2.12]{Forbush}}] \label{circ-orient}
  Let $D$ be the terminal diagonal of a hamiltonian path~$\hampath$ in~$\Bmn$,
with initial square~$\init$ and terminal square~$\term$, and let $|D|$ denote
the cardinality of~$D$.
	\begin{enumerate}
	
	\item \label{circ-orient-term}
	For each $\sigma \in D$, there is a unique integer $u(\sigma) \in
\{1,2,\ldots,|D|\}$ with $\sigma = \term (N E^{-1})^{u(\sigma)}$; the
square~$\sigma$ travels east iff $u(\sigma) < u(\init E^{-1})$.

	\item \label{circ-orient-init}
	Similarly, there is a unique integer $v(\sigma) \in
\{1,2,\ldots,|D|\}$ with $\sigma = \init E^{-1} (E N^{-1})^{v(\sigma)}$; the
square~$\sigma$ travels east iff $v(\sigma) < v(\term)$.
	
	\end{enumerate}
 \end{cor}

\begin{cor}[{}{\cite[Cor.~2.13]{Forbush}}] \label{uniqueness}
 A hamiltonian path is uniquely determined by specifying
 \begin{enumerate}
 \item \label{initial-choice} its initial square;
 \item \label{terminal-choice} its terminal square; and
 \item \label{non-terminal-choice} which of its non-terminal diagonals travel
east.
 \end{enumerate}
 \end{cor}

\subsection{Further restrictions on hamiltonian paths}

\begin{prop}[{}{\cite[Thm.~3.2]{Forbush}}] \label{00NotInit}
If $m,n\ge 3$, then $(0,0)$ is not the initial square of any hamiltonian path in $\Bmn$. Therefore, $D_{m+n-2}$
is not the terminal diagonal of any hamiltonian path\/ \textup(and $\Bmn$ does not have a hamiltonian cycle\/\textup).
\end{prop}

\begin{lem}[{}{\cite[Lem.~3.4]{Forbush}}] \label{OuterEastOK}
Suppose $\subdiag_a \cup \subdiag_b$ is the terminal diagonal of a hamiltonian path~$\hampath$ in~$\Bmn$, with $m,n \ge 3$ and $a \le b$. Choose $(p,q) \in \subdiag_{b+1}$, and let $P$ be
the unique path in~$\hampath$ that starts at $(p,q)$ and ends in~$\subdiag_a$, without
passing through $\subdiag_a$. Then the terminal square of~$P$ is the inverse of $(p,q)$.
 \end{lem}

\begin{defn}
Let $S_a \cup S_b$ be the terminal diagonal of a hamiltonian path, with $a \le b$, and let $S_i \cup S_j$ be some other diagonal of~$\Bmn$, with $i \le j$ and $i + j < m + n - 2$. We say that:
	\begin{enumerate}
	\item $S_i \cup S_j$ is an \emph{outer diagonal} if $i < a$ (or, equivalently, $j > b$).
	\item $S_i \cup S_j$ is an \emph{inner diagonal} if $i > a$ (or, equivalently, $j < b$).
	\end{enumerate}
\end{defn}

\Cref{OuterEastOK} has the following important consequence.

\begin{cor}[cf.\ {\cite[Thm.~3.5]{Forbush}}] \label{OuterCan}
 Assume that $\hampath$ is a hamiltonian path from~$\init$ to~$\term$ in~$\Bmn$, with $m \ge n \ge 3$. Define $\hampath_E$ and~$\hampath_N$ to be the subdigraphs of~$\Bmn$, such that%
\noprelistbreak
	\begin{itemize}
	\item $\init$ has invalence~$0$, but the invalence of all other squares is~$1$ in both $\hampath_E$ and~$\hampath_N$,
	\item $\term$~has outvalence~$0$, but the outvalence of all other squares is~$1$ in both $\hampath_E$ and~$\hampath_N$,
	\item each inner diagonal travels exactly the same way in~$\hampath_E$ and~$\hampath_N$ as it does in~$\hampath$,
	and
	\item each outer diagonal travels east in~$\hampath_E$, but travels north in~$\hampath_N$.
	\end{itemize}
Then:
	\begin{enumerate}
	\item \label{OuterCan-East}
	$\hampath_E$ is a hamiltonian path from~$\init$ to~$\term$. \refnote{OuterCanPf}
	\item \label{OuterCan-North}
	$\hampath_N$ is a hamiltonian path from~$\init$ to~$\term$ if and only if the diagonal $S_{n-1} \cup S_{m-2}$ is not outer.
	\end{enumerate}
\end{cor}

\section{Hamiltonian paths in which all non-terminal diagonals travel north} \label{AllNorth}

We eventually need to understand all of the hamiltonian paths in~$\Bmn$, but this \lcnamecref{AllNorth} considers only the much simpler special case in which every non-terminal diagonal is required to travel north. Although this may seem to be a very restrictive assumption, \cref{stretch} below 
will allow us to obtain the general case from this one.

\begin{prop} \label{Init=TermE}
Assume $S_a \cup S_b$ is the terminal diagonal of a hamiltonian path~$\hampath$ in~$\Bmn$, with $m \ge n \ge 3$ and $a \le b$.  Let $\term_+$ be the southeasternmost square in~$S_b$. If all non-terminal diagonals travel north in~$\hampath$, then $a \le m-2 \le b$, and $\term_+E$ is the initial square of either $\hampath$ or the inverse of~$\hampath$ \textup(or the transpose or transpose-inverse of~$\hampath$, if $m = n = a + 2 = b + 1$\textup), unless $a + 1 = b = n = m-2$, in which case the initial square \textup(of either $\hampath$ or~$\widetilde\hampath$\textup) might also be $\term_+ = (n, 0)$.
\end{prop}

\begin{proof}
From \fullcref{OuterCan}{North}, we know that $a \le n-1$ and $m - 2 \le b$. Since $n \le m$, this immediately implies $a \le m - 2 \le b$, unless $a = n - 1 = m - 1$. But then $b = m + n - 3 - a = m - 2 < a$, which contradicts the fact that $a \le b$.

For convenience, write $\term_+ = (x,y)$, and suppose the initial square is not as described.
We consider two cases. 

\begin{case} \label{Init=TermEPf=x=m-1}
Assume $x = m-1$. 
\end{case}
Note that $\term_+E = (0, n - 1 - y)$ is the inverse of~$\term_+$, so $\term_+$ cannot be the terminal square of~$\hampath$ (since $\term_+E$ is not the initial square of the inverse of~$\hampath$). 

Assume, for the moment, that $\term_+E$ is not in the terminal diagonal. Then, by assumption, $\term_+E$ travels north. So $\term_+$ cannot travel east. (Otherwise, the hamiltonian path~$\hampath$ would contain the cycle $[\term_+](E N^{2y+1})$.) Therefore, since $\term_+$ is not the terminal square of~$\hampath$, we conclude that $\term_+$ travels north. Since $\term_+E$ is not the initial square (and must therefore be entered from either $\term_+$ or $\term_+ EN^{-1}$, we conclude that $\term_+ EN^{-1}$ travels north. So $\hampath$ contains the cycle $[\term_+](N^{2n})$. This is a contradiction.

We may now assume that $\term_+E$ is in the terminal diagonal. 
However, $\term_+ E N^{-1}$ is also in the terminal diagonal (since it is obviously in the same diagonal as~$\term_+$, which is in the terminal diagonal). It follows that $\term_+ E N^{-1} \in S_a$ \refnote{Sa+1}
and $\term_+ E \in S_b$, with $b = a + 1$. 
Since
	$$b = x + y = (m - 1) + y \ge m-1 ,$$
this implies
	$$2m - 3 \ge m + n - 3 = a + b 
	= 2b - 1 \ge 2(m-1) - 1 = 2m - 3 .$$
Therefore, we must have equality throughout both strings of inequalities, so 
	$$ \text{$y = 0$, \ $m = n$, \ $b = m-1$, \ and \ $a = m-2$} .$$
(Since $m = n$, the desired contradiction can be obtained from \cite[Thm.~3.12]{Forbush}, but, for completeness, we provide a direct proof.)
Since $(m-1,0) = (m-1,y) = \term_+$ is not the terminal square, it must travel either north or east. We consider these two possibilities individually.

Assume, for the moment, that $(m-1,0)$ travels east (to $(0,m-1)$, because $m = n$). Clearly, $(0,m-1)$ does not travel north (because $\hampath$ does not contain the cycle $[(m-1,0)](EN)$). Also, $(0,m-1) = \term_+E$ is not the terminal square (because it is not the initial square of the transpose-inverse of~$\hampath$). So $(0,m-1)$ must travel east. Since $(0,m-1)N = (m-1,0)$ is not the initial square (because $(0,m-1) = \term_+E$ is not the initial square of the transpose of~$\hampath$), we conclude from \cref{term-travels} that $(m-2,0) = (0,m-1)NE^{-1}$ also travels east. And $(1,m-1)$ travels north, because it is not in the terminal diagonal. So $\hampath$ contains the cycle 
	$$ (0,m-1) \stackrel{E}{\to} (1,m-1) \stackrel{N}{\to} (m-2, 0) \stackrel{E}{\to} (m-1, 0) \stackrel{E}{\to}  (0, m-1)  .$$
This is a contradiction.

We may now assume that $(m-1,0)$ travels north. Since $(m-1,0)E = \term_+E$ is not the initial square, we conclude from \cref{term-travels} that $(0,m-2) = (m-1,0)EN^{-1}$ also travels north. By applying the same argument to the transpose of~$\hampath$, \refnote{ApplyToTranspose}
we see that $(0,m-1)$ and $(m-2,0)$ must travel east. Also, $(1,m-1)$ travels north, because it is not in the terminal diagonal. So $\hampath$ contains the cycle 
	$$ (m-1,0) 
	\stackrel{N^{2m-2}}{\longrightarrow} (0,m-2) 
	\stackrel{N}{\to} (0,m-1)
	\stackrel{E}{\to} (1,m-1)
	\stackrel{N}{\to} (m-2,0)
	\stackrel{E}{\to} (m-1,0) 
	.$$

This contradiction completes the proof of \cref{Init=TermEPf=x=m-1}.

\begin{case}
Assume $x < m-1$. 
\end{case}
Since $(x,y) = \term_+$ is the southeasternmost square in~$S_b$, we must have $y = 0$ (otherwise, $(x + 1, y - 1)$ is a square in~$S_b$ is farther southeast), so $b = x < m-1$. Since we know from the first sentence 
of the proof that $m - 2 \le b$, we conclude that $b = m - 2$ (and $a = m + n - 3 - b = n - 1$). Therefore
	$$ \text{$\term_+ = (m-2,0)$, \ so \ $\term_+E = (m-1,0)$} .$$
 
Note that $(0,n-1)$ cannot travel north (otherwise, $\hampath$ contains the cycle $(N^{2n})$, since $(0,n-1)$ is the only square of this cycle that is in the terminal diagonal). Also, $(0,n-1)$ is not the terminal square (since $(m-1,0) = \term_+E$ is not the initial square of the inverse of~$\hampath$). Therefore $(0,n-1)$ must travel east. Then, since $(0,n-1)N = (m-1,0) = \term_+E$ is not the initial square, we know that $(m-2,0) = \term_+$ also travels east. 

Since $\hampath$ cannot contain the cycle
	$$ (m-1,0) 
	\stackrel{N^{2n-1}}{\longrightarrow} (0,n-1)
	\stackrel{E}{\to} (1,n-1)
	\stackrel{N}{\to} (m-2,0)
	\stackrel{E}{\to} (m-1,0)
	, $$
we know that $(1,n-1)$ does not travel north. Therefore this square is in the terminal diagonal, which means  $b = 1 + (n-1) = n$, so we have $a + 1 = n = b = m-2$. Hence, we are in the exceptional case at the end of the statement of the \lcnamecref{Init=TermE}. Therefore, $(1,n-1)$ is not the terminal square of~$\hampath$ (since $(m-2,0) = \term_+$ is not the initial square of the inverse of~$\hampath$).
Since we have already seen that it does not travel north, we conclude that $(1,n-1)$ travels east. 

Applying the argument of the preceding paragraph 
to the inverse of~$\hampath$ tells us that $(1,n-1)$ travels east in~$\widetilde\hampath$. Taking the inverse, this means $(m-2,0)$ travels east in~$\hampath$. Also, we know that $(2,n-1)$ travels north (because it is not in the terminal diagonal), and we know that $(m-3,0)$ travels east (because $(m-3,0)EN^{-1} = (1,n-1)$ travels east and $(m-3,0)E = \term_+$ is not the initial square). Therefore, $\hampath$ contains the cycle 
	$$ (m-1,0) \stackrel{N^{2n-1}}{\longrightarrow} (0,n-1) \stackrel{E^2}{\to} (2,n-1) \stackrel{N}{\to} (m-3,0)  \stackrel{E^2}{\to} (m-1,0) .$$ 
This is a contradiction.
\end{proof}

The above \lcnamecref{Init=TermE} usually allows us to assume that the initial square of a hamiltonian cycle is~$\term_+ E$ (if all non-terminal diagonals travel north). The following result finds the possible terminal squares in this case.

\begin{prop} \label{UsualEndpt}
Let 
\noprelistbreak
	\begin{itemize}
	\item $m \ge n \ge 3$,
	\item $S_a \cup S_b$ be a diagonal in~$\Bmn$,
	\item $\term_+$ be the southeasternmost square in~$S_b$,
	and
	\item $\term$ be any square in $S_a \cup S_b$. 
	\end{itemize}
There is a hamiltonian path~$\hampath$ from $\term_+ E$ to~$\term$ in which all non-terminal diagonals travel north if and only if $\term$~is either $(\mm,a - \mm)$ or $(\mf, b - \mf)$ \textup(and $\term = (\mm,a - \mm)$ if $a = b$\textup), and $a \le m-2 \le b$.
\end{prop}

\begin{proof}
Let $\sigma_a = (\mm, a - \mm)$ and $\sigma_b = (\mf, b - \mf)$.

\smallskip

$({\Rightarrow})$
\fullCref{OuterCan}{North} tells us that $a \le m-2 \le b$. 
(See the first paragraph 
of the proof of \cref{Init=TermE}.) This establishes one conclusion of the \lcnamecref{UsualEndpt}.

We now wish to show that $\term$ is either $\sigma_a$ or~$\sigma_b$, and that $\term = \sigma_a$ if $a = b$. Assume the contrary.

Note that $\term_+$ must travel north in~$\hampath$, since $\term_+ E$ is the initial square (and $\Bmn$ does not have a hamiltonian cycle). \refnote{TauPlusMustNorth}

\setcounter{case}{0}

\begin{case}
Assume $m$~is odd. 
\end{case}
Note that $\mm = \mf$ in this case (and we have $m - 1 - \mf = \mf$).
Since $\hampath$ cannot contain the cycle $[(\mf,0)](N^n)$, we know that some square in this cycle does not travel north in~$\hampath$. This square must be in the terminal diagonal, so it is either $\sigma_a$ or~$\sigma_b$. It is therefore not the terminal square, so it must travel east.

From the preceding paragraph, we see that the square $\sigma_b$ must exist in~$\Bmn$, \refnote{sigmaBisSquare}
so $b - \mf \le n - 1$. Therefore 
	\begin{align*}
	 a - \mf + 1 
	 &= (m + n - 3) - b - \mf + 1 
	 \\&\ge (m + n - 3) - (\mf + n - 1) - \mf + 1 
	 \\&=  0 
	 . \end{align*}
Also, 
	$$ a - \mf + 1 \le b - \mf + 1\le (n-1) + 1 = n ,$$
so $a - \mf + 1 \le n - 1$ unless $a = b = \mf + n - 1$. But the alternative yields a contradiction:
	$$ m + n - 3 = a + b = 2(\mf + n - 1) = m + 2n - 3 > m + n - 3 .$$
Therefore, the square $(\mf - 1, a - \mf + 1)$ exists (and is in~$S_a$). 

Suppose $\sigma_b$ travels east. Since $\term_+ E$ is the initial square, we see from \cref{circ-orient} that $(\mf - 1, a - \mf + 1)$ travels east. \refnote{sigmaAGoesEast}
If $a < b$, this implies that $\hampath$ contains the cycle 
	$$ \sigma_b \stackrel{E N^{n - b + a - \mf + 1}}{\longrightarrow} (\mf-1, a - \mf + 1) \stackrel{EN^{b - a - 1}}{\longrightarrow} \sigma_b  .$$
On the other hand, if $a = b$, this implies that $\hampath$ contains the cycle
	$$ \sigma_b \stackrel{E N^{n + 1}}{\longrightarrow} (\mf-1, b - \mf + 1) \stackrel{E N^{n - 1}}{\longrightarrow} \sigma_b  .$$
In either case, we have a contradiction. 

We may now assume that $\sigma_b$ travels north. So it must be $\sigma_a$ that travels east (and $\sigma_a \neq \sigma_b$, so $a \neq b$). From \cref{circ-orient}, we see that \refnotelower{AEastBNorth}
	$$ \text{$(\mf-1, a - \mf + 1)$ travels east and $(\mf+1, b - \mf-1)$ travels north} .$$
So $\hampath$ contains the cycle 
	\begin{align*}
	 \sigma_a 
	&\stackrel{E}{\to} (\mf + 1, a - \mf)
	\stackrel{N^{b - a - 1}}{\longrightarrow} (\mf + 1, b - \mf - 1) 
	\\&\stackrel{N^{n - b + a +2}}{\longrightarrow} (\mf - 1, a - \mf + 1) 
	\stackrel{E}{\to} (\mf, a - \mf + 1) 
	\stackrel{N^{b - a - 1}}{\longrightarrow} \sigma_b
	\stackrel{N^{n - b + a}}{\longrightarrow} \sigma_a
	. \end{align*}
This is a contradiction.

\begin{case}
Assume $m$~is even. 
\end{case}
Note that $\mm = \mf - 1$ in this case (and we have $m - 1 - \mf = \mf - 1$).
Since $\hampath$ does not contain the cycle $[(\mf,0)](N^{2n})$, we know that some square in this cycle does not travel north. In other words, there is a square $(x,y)$ that does not travel north, such that $x \in \{\mf - 1, \mf\}$.

Assume, for the moment, that $b - \mf = n - 1$. Then $\sigma_b$ is the only square that is in the intersection of the terminal diagonal with the cycle $[(\mf,0)](N^{2n})$, \refnote{TermDiagAndCycle}
so it must be $\sigma_b$ that does not travel north. Since, by assumption, $\sigma_b$~is not the terminal square, we conclude that $\sigma_b$ travels east.
Then \cref{circ-orient} implies that $ (\mf - 2, 0) = \sigma_b N E^{-1}$ also travels east. So $\hampath$ contains the cycle 
	$$ \sigma_b 
	\stackrel{E}{\to} (\mf + 1, n - 1) 
	\stackrel{N}{\to} (\mf - 2, 0)  
	\stackrel{E}{\to} (\mf - 1, 0)  
	\stackrel{N^{2n-1}}{\longrightarrow} 
	\sigma_b 
	.$$
This is a contradiction.

We may now assume $b - \mf \le n - 2$, so 
	$$ a = (m + n - 3) - b \ge (m + n - 3) - (n - 2 + \mf) = m - 1 - \mf = \mm .$$
Therefore $S_a$ contains the square $(\mm, a - \mm) = \sigma_a$.
Note that $\sigma_a$ cannot travel north. (Otherwise, \cref{circ-orient} implies that $\sigma_b$ also travels north, \refnote{SigmaBAlsoNorth}
contrary to the fact that at least one of these two squares does not travel north.) Since, by assumption, $\sigma_a$ is not the terminal square, we conclude that $\sigma_a$ travels east. 

Since $\hampath$ does not contain the cycle $[\sigma_a](E, N^n)$, we conclude that $a \neq b$ and $\sigma_b$ does not travel north. Therefore $\sigma_b$ travels east. Then \cref{circ-orient} tells us that $(\mf - 1,b-\mf + 1)$ and every square in~$S_a$ all travel east. \refnote{LotsEast}
So $\hampath$ contains the cycle 
	\begin{align*}
	\sigma_b
	&\stackrel{E}{\to} (\mf + 1, b - \mf)
	\stackrel{N^{n-b + a + 2}}{\longrightarrow} (\mf-2,a - \mf + 2)
	\\& \stackrel{E}{\to} (\mf-1,a - \mf + 2)
	\stackrel{N^{b-a-1}}{\longrightarrow} (\mf - 1,b-\mf + 1)
	\\&\stackrel{E}{\to} (\mf,b-\mf + 1)
	\stackrel{N^{n-b + a}}{\longrightarrow} \sigma_a
	\stackrel{E}{\to} (\mf, a - \mf + 1)
	\stackrel{N^{b-a-1}}{\longrightarrow} \sigma_b 
	. \end{align*}
This is a contradiction.

\medbreak

$({\Leftarrow})$
We use $(\dots)^k$ to represent the concatenation of $k$~copies of the sequence $(\dots)$. (For example, $(N^3,E)^2 = (N,N,N,E,N,N,N,E)$.)

If $\sigma_a$ exists (that is, if $a \ge \mm$), then we have the following hamiltonian path $\hampath_a$ from $\term_+ E$ to~$\sigma_a$ \csee{HamaFig}:
	$$ [\term_+ E] \bigl( (N^{2n-1}, E)^{\mm}, N^{2n-1} \bigr) . $$

\begin{figure}[b]
\hrule\bigskip
\centerline{
\includegraphics{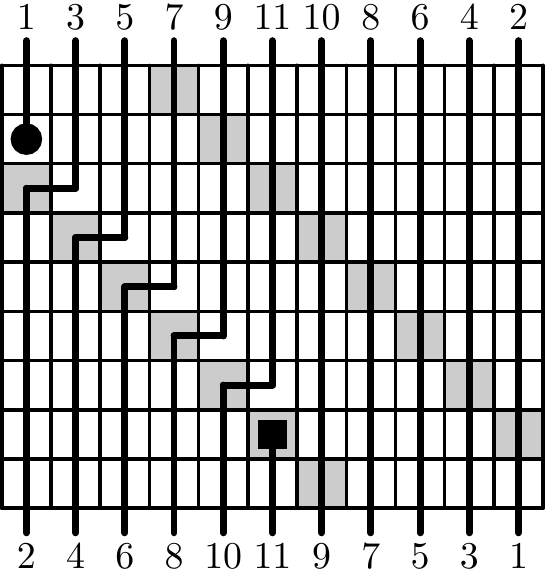}
\hfil
\vbox{\hbox{\includegraphics{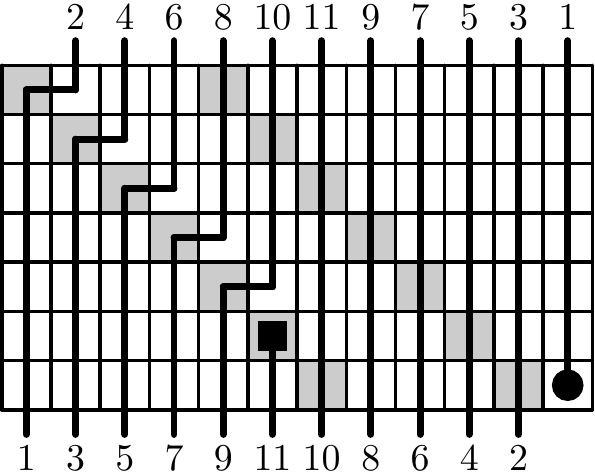}} \vskip 0.5cm}
}
\caption{Illustrative examples of the hamiltonian path~$\hampath_a$ from~$\term_+E$ (\!\lower2pt\hbox{\larger[4]$\bullet$}\!) to $\sigma_a$ (\vrule width 7pt height 6.5pt depth 0.5pt). (The terminal diagonal is shaded.)}
\label{HamaFig}
\end{figure}

Now assume $\sigma_b$ exists (that is, $b \le \mf + n - 1$) and $a \neq b$. 
	\begin{itemize}
	
	\item If $m$ is odd, then we have the following hamiltonian path $\hampath_b$ from $\term_+ E$ to~$\sigma_b$ \csee{HambOddFig}:
	\begin{align*}
	 [\term_+ E] \bigl( &(N^{2n-1}, E )^{b - n + 1}, 
		\\&(N^{b - a - 1}, E, N^{2i-1}, E, N^{n + a - b - 2i + 1}, E)_{i=1}^{\mf + n - b - 1}, 
		   \\&N^{b-a-1} \bigr)
	. \end{align*}

\begin{figure}[b]
\hrule\bigskip
\centerline{
\includegraphics{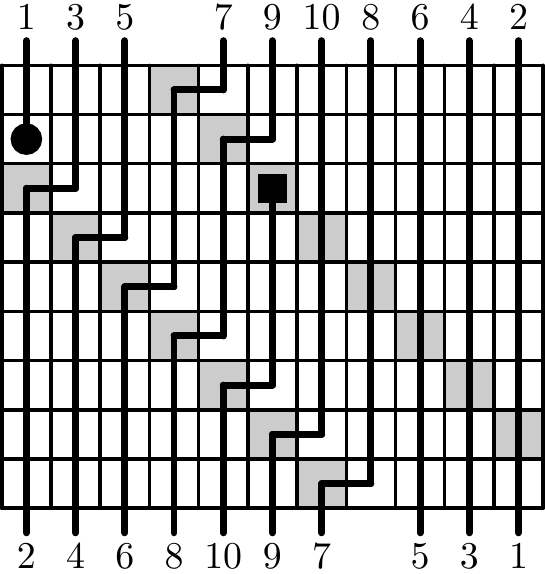}
\hfil
\vbox{\hbox{\includegraphics{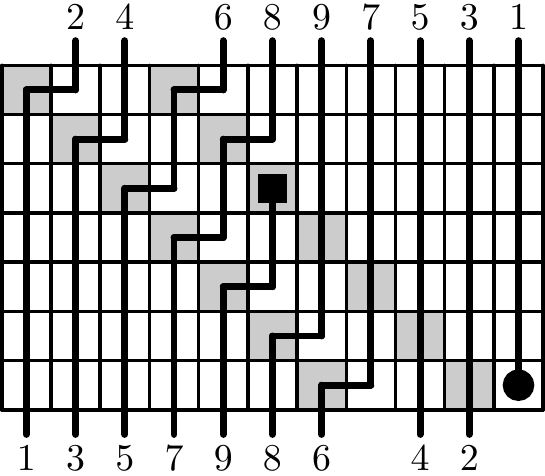}} \vskip 0.5cm}
}
\caption{Illustrative examples of the hamiltonian path~$\hampath_b$ from~$\term_+E$ (\!\lower2pt\hbox{\larger[4]$\bullet$}\!) to $\sigma_b$ (\vrule width 7pt height 6.5pt depth 0.5pt) when $m$ is odd. (The terminal diagonal is shaded.)}
\label{HambOddFig}
\end{figure}
	
	\item If $m$ is even, then we have the following hamiltonian path $\hampath_b$ from $\term_+ E$ to~$\sigma_b$ \csee{HambEvenFig}:
	\begin{align*}
	 [\term_+ E] \bigl( &(N^{2n-1}, E )^{b - n + 1}, 
		\\&(N^{b - a - 1}, E, N^{2i-1}, E, N^{n + a - b - 2i + 1}, E)_{i=1}^{\mf + n - b - 2}, 
		   \\& N^{b-a-1}, E, N^{n+a-b}, E, N^{b-a-1} \bigr)
	. 
	\qedhere \end{align*}

	\end{itemize}
\end{proof}

\begin{figure}[b]
\hrule\bigskip
\centerline{
\includegraphics{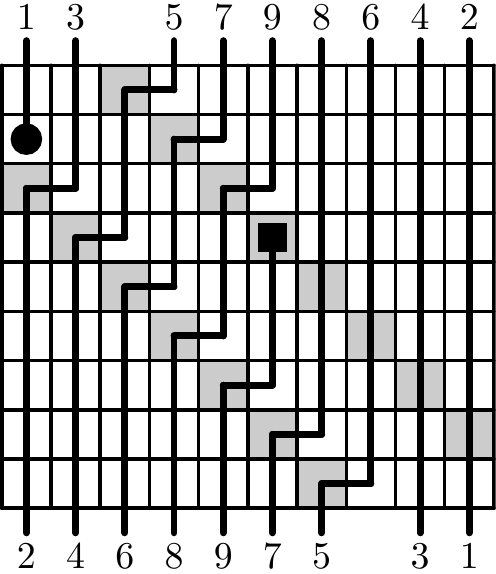}
\hfil
\vbox{\hbox{\includegraphics{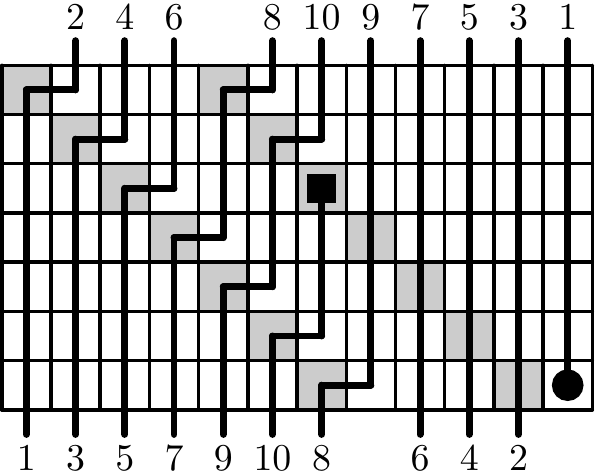}} \vskip 0.5cm}
}
\caption{Illustrative examples of the hamiltonian path~$\hampath_b$ from~$\term_+E$ (\!\lower2pt\hbox{\larger[4]$\bullet$}\!) to $\sigma_b$ (\vrule width 7pt height 6.5pt depth 0.5pt) when $m$ is even. (The terminal diagonal is shaded.)}
\label{HambEvenFig}
\end{figure}

We conclude this \lcnamecref{AllNorth} by finding the terminal square in the exceptional case that is at the end of the statement of \cref{Init=TermE}:

\begin{lem} \label{init=tauplus}
Let $\term$ be any square in $\board{m}{m-2}$, with $m \ge 5$. 
There is a hamiltonian path~$\hampath$ from $(m-2,0)$ to~$\term$ in which all non-terminal diagonals travel north if and only if $\term = (\mf - 1, \mm - 1)$.
\end{lem}

\begin{proof}
Let $n = m - 2$, $a = m - 3 = n - 1$ and $b = a + 1 = n$. 

$({\Rightarrow})$
Since the initial square is $(m-2,0) = (a,0)E$, we know from \cref{iotaE=tau} that $S_a$ is a terminal subdiagonal. Then the other part of the terminal diagonal is $S_{m + n - 3 - a} = S_b$.

Since $\hampath$ does not contain the cycle $[(m-1,0)](N^{2n})$, we know that $(0,n-1)$ does not travel north. From \cref{circ-orient}, this implies that the terminal square~$\term$ is somewhere in~$S_a$, and that every square in~$S_b$ travels east. 

There are no inner diagonals (since $b = a + 1$), so, by \cref{OuterEastOK}, we may let $\hampath'$ be the hamiltonian path from $(m-2,0)$ to~$\term$ in which all non-terminal diagonals travel east. 

\setcounter{case}{0}

\begin{case}
Assume $m$ and~$n$ are odd.
\end{case}
Since $\hampath'$ does not contain the cycle $[(0,\nf)](N^m)$, we know that $(\nf,\nf)$ does not travel east. We may also assume it is not the terminal square, for otherwise $\term = (\nf,\nf) = (\mf - 1, \mm - 1)$ (since $\nf = \nm$ and $n = m - 2$), as desired. So $(\nf,\nf)$ travels north. From \cref{circ-orient}, we conclude that $(\nf + 1,\nf - 1)$ also travels north. \refnote{AlsoNorth}
So $\hampath'$ contains the cycle
	$$ (\nf,\nf)
		\stackrel{N}{\to} (\nf, \nf + 1)
		\stackrel{E^{m+1}}{\to} (\nf + 1, \nf - 1)
		\stackrel{N}{\to} (\nf + 1, \nf)
		\stackrel{E^{m-1}}{\to} (\nf, \nf)
	. $$
This is a contradiction.

\begin{case}
Assume $m$ and~$n$ are even.
\end{case}
Since $\hampath'$ does not contain the cycle $[(0,\nf)](E^{2m})$, we know that $(\nf,\nm)$, $(\nf,\nf)$, $(\nm,\nf)$, and $(\nm, \nf+1)$ do not all travel east. From \cref{circ-orient}, we conclude that $(\nf,\nm)$ does not travel east. \refnote{CantEast} We may also assume it is not the terminal square, for otherwise $\term = (\nf,\nm) = (\mf - 1, \mm - 1)$, as desired. So $(\nf,\nm)$ travels north. Then $\hampath'$ contains the cycle
	$$ (\nf,\nm)
		\stackrel{N}{\to} (\nf, \nf)
		\stackrel{E^m}{\to} (\nf, \nm)
	. $$
This is a contradiction.

\medbreak

$({\Leftarrow})$ We have the following hamiltonian path from $(m-2,0)$ to~$(\mf - 1, \mm - 1)$ in~$\Bmn$ when $n = m-2$ \csee{Ham(m-2:0)Fig}:
	$$ \begin{cases}
	[(m-2,0)] \bigl( (E, N^{2n+1-2i}, E^2, N^{2i})_{i=1}^{\nf - 1}, E, N^n \bigr) & \text{if $m$ and~$n$ are odd}, \\
	[(m-2,0)] (E, N^{2n+1-2i}, E^2, N^{2i})_{i=1}^{\nf}\# & \text{if $m$ and~$n$ are even}
	\end{cases} $$
(where $\#$ indicates deletion of the last term of the sequence).
\end{proof}

\begin{figure}[b]
\hrule\bigskip
\centerline{
\includegraphics{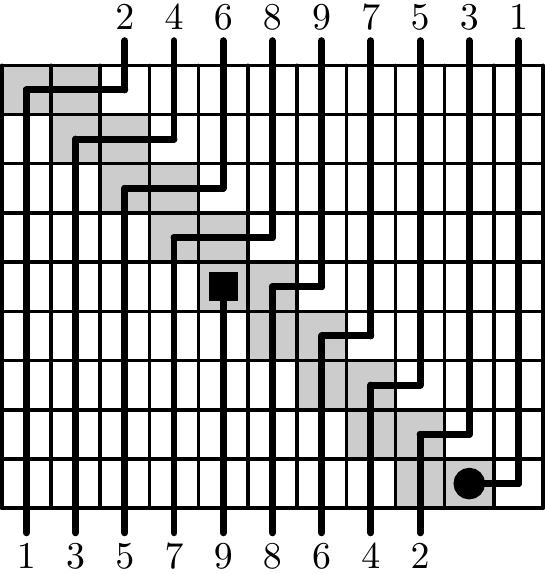}
\hfil
\vbox{\hbox{\includegraphics{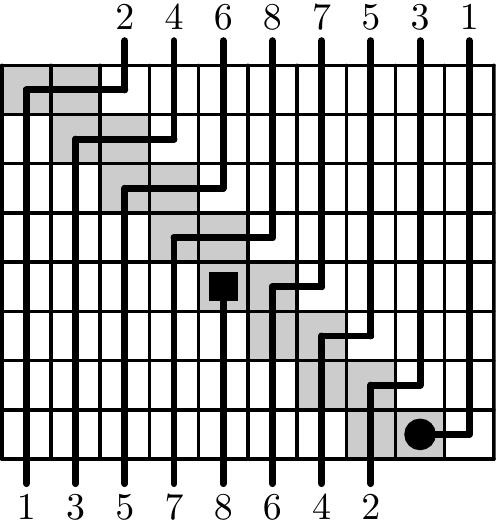}} \vskip 0.25cm}
}
\caption{Illustrative examples (with $m$ odd on the left, and $m$~even on the right) of the hamiltonian path from $(m-2,0)$ (\!\lower2pt\hbox{\larger[4]$\bullet$}\!) to~$(\mf - 1, \mm - 1)$ (\vrule width 7pt height 6.5pt depth 0.5pt) in $\board{m}{m-2}$, such that all non-terminal diagonals travel north. (The terminal diagonal is shaded.)}
\label{Ham(m-2:0)Fig}
\end{figure}

\section{Reduction to diagonals that travel north} \label{ReductionSect}

\begin{defn}
A diagonal $S_i \cup S_j$ of~$\Bmn$ with $i \le j$ is said to be \emph{rowful} if $n-1 \le i \le j \le m-2$. (In other words, $S_i \cup S_j$ is rowful if $S_i$ and~$S_j$ each contain a square from every row of the checkerboard.) The subdiagonals $S_i$ and $S_j$ of a rowful diagonal are also said to be \emph{rowful}.
\end{defn}

\Cref{IgnoreEast} below 
shows that if a rowful diagonal travels east, then it basically just stretches the checkerboard to make it wider \csee{RowfulEFig}. \Cref{stretch} uses this observation to show that finding a hamiltonian path between any two given squares of~$\Bmn$ reduces to the problem of finding a hamiltonian path in a smaller checkerboard, such that all non-terminal diagonals travel north. 

\begin{figure}[b]
\hrule\bigskip
\centerline{
\includegraphics{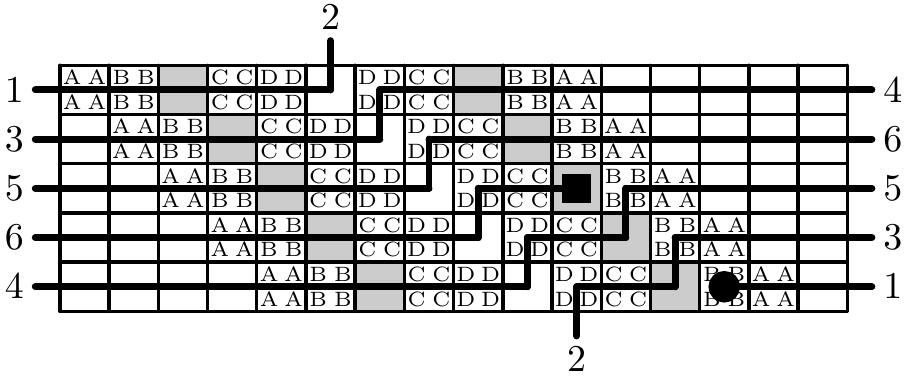}
}
\centerline{\raise0.7in\hbox to 0pt{\hss\Large$\leadsto$ \quad}
\includegraphics{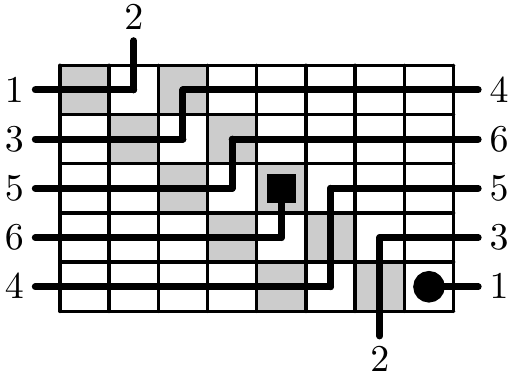}
}
\caption{The diagonals marked 
\lower3pt\vbox{\tiny\hbox{A\,A}\hbox{A\,A}},
\lower3pt\vbox{\tiny\hbox{B\,B}\hbox{B\,B}},
\lower3pt\vbox{\tiny\hbox{C\,C}\hbox{C\,C}},
\lower3pt\vbox{\tiny\hbox{D\,D}\hbox{D\,D}}
are rowful and travel east. Removing them yields a hamiltonian path in a smaller checkerboard. (As usual, the terminal diagonal is shaded.)}
\label{RowfulEFig}
\end{figure}

\begin{warn}
The subdiagonal $S_{m-1}$ is \emph{not} rowful, even though it contains a square from every row of the checkerboard (if $m \ge n$), because it is a constituent of the diagonal $S_{n-2} \cup S_{m-1}$, which is not rowful.
\end{warn}

\begin{notation}
For $i,j \in \NN$, define $\Delta_{i,j} \colon \NN \to \{0,1,2\}$ by
	$$ \Delta_{i,j}(k) = \bigl| \{i,j\} \cap \{0,1,2\ldots, k-1\} \bigr| .$$
Then, for each square $(p,q)$ of~$\Bmn$, we let 
	$$\Delta_{i,j}^\square(p,q) = \bigl( p - \Delta_{i,j}(p+q), q \bigr) .$$
\end{notation}

\begin{lem} \label{IgnoreEast}
Suppose
\noprelistbreak
	\begin{itemize}
	\item $\term_0$ and $\term$ are two squares of~$\Bmn$ that are in the same diagonal,
	and
	\item $S_i \cup S_j$ is a rowful diagonal of~$\Bmn$ that is not the diagonal containing $\term_0$ and~$\term$.
	\end{itemize}
Then there is a hamiltonian path~$\hampath$ from $\term_0 E$ to~$\term$ in~$\Bmn$, such that $S_i \cup S_j$ travels east, if and only if there is a hamiltonian path~$\hampath'$ from $\bigl( \Delta_{i,j}^\square(\term_0) \bigr) E$ to $\Delta_{i,j}^\square(\term)$ in $\board{m-\Delta_{i,j}(m+n)}{n}$.

More precisely, if $\sigma$ is any square of~$\Bmn$ that is not in $S_i \cup S_j$, then the square $\Delta_{i,j}^\square(\sigma)$ travels the same direction in~$\hampath'$ as the square $\sigma$ travels in~$\hampath$.
\end{lem}

\begin{proof}
Assume, for the moment, that $j \neq i + 1$ (so $S_i E \cap S_j = \emptyset$).
Define a digraph~$\boardonly'$ from~$\Bmn$ by 
	\begin{enumerate}
	\item replacing each directed edge $\sigma \to \phi$, such that $\phi \in S_i \cup S_j$, with a directed edge from $\sigma$ to~$\phi E$,
	and
	\item \label{DeleteSquares}
	deleting all the squares in $S_i \cup S_j$ (and the incident edges).
	\end{enumerate}
It is clear that hamiltonian paths in~$\boardonly'$ correspond to hamiltonian paths in~$\Bmn$ such that $S_i \cup S_j$ travels east. Since the digraph~$\boardonly'$ is isomorphic to~$\board{m-\Delta_{i,j}(m+n)}{n}$ (via the map~$\Delta_{i,j}^\square$), the desired conclusion is immediate. \refnote{StretchInitialSquare}

If $j = i + 1$, then the definition of~$\boardonly'$ needs a slight modification: instead of considering only a directed edge $\sigma \to \phi$, one needs to allow for the possibility of a longer path. Namely, if there is a path $\sigma \to \phi \stackrel{E}{\to} \alpha$ with $\phi \in S_i$ and $\alpha \in S_j$, then, instead of inserting the edge $\sigma$ to~$\phi E$ (which cannot exist in~$\boardonly'$, because $\phi E = \alpha$ is one of the squares deleted in~\pref{DeleteSquares}), one inserts the edge $\sigma \to \alpha E$, because $\hampath$ must proceed from~$\sigma$ to~$\alpha E$ (via $\phi$ and~$\alpha$) if it travels from~$\sigma$ to~$\phi$.
\end{proof}

When $m = n$, it was proved in \cite[Prop.~3.3]{Forbush} that if some inner diagonal travels east in a hamiltonian path, then all inner diagonals must travel east. That is not always true when $m \neq n$, but we have the following weaker statement:

\begin{lem}[{}{cf.\ \cite[Prop.~3.3]{Forbush}}] \label{CanGoEast}
Let $\hampath$ be a hamiltonian path in $\Bmn$ with $m \ge n$. If $D$ is any inner diagonal that travels east in $\hampath$, then either $D$~is rowful, or all inner diagonals travel east.
\end{lem}

\begin{proof}
By repeated application of \cref{IgnoreEast}, we may assume that all rowful diagonals travel north. In this situation, we wish to show that if some inner diagonal $S_i \cup S_j$ travels east, then all inner diagonals travel east. Assume $j$~is minimal (or, equivalently, that $i$~is maximal), such that $S_i \cup S_j$ is an inner diagonal that travels east, and $i \le j$. This means $S_{i+1}, S_{i+2}, \ldots, S_{j-1}$ all travel north. 
From the first sentence 
of the proof, we know that  $S_i \cup S_j$ is not rowful, so $j \ge m - 1$. Therefore, we may let $\sigma = (m-1, j - m + 1)$. Since $\sigma \in S_j$ and the first coordinate of~$\sigma$ is $m-1$, we see that $\sigma E N^{-1} \in S_i$, so $\sigma E \in S_{i + 1}$. So $\hampath$ contains the cycle $[\sigma](E, N^{2(j - m) + 3})$. This is a contradiction.
\end{proof}

The following result essentially reduces the proof of \cref{InitSquares} to the special case considered in \cref{ReductionSect}, where all non-terminal diagonals travel north. (Although the diagonals travel east in conclusion \pref{stretch-HP-east}, passing to the transpose yields a hamiltonian path in which all non-terminal diagonals travel north, because the checkerboard~$\board{m'}{n}$ is square in this case.)

\begin{prop} \label{stretch}
Assume
\noprelistbreak
	\begin{itemize}
	\item $S_a \cup S_b$ is a diagonal of $\Bmn$, with $m \ge n$, $a \le b$, and $a + b \neq m + n - 2$,
	\item $(x,y)$ is a square in $S_a \cup S_b$,
	\item $(p,q)$ is a square of $\Bmn$ with $p + q - 1 \in \{a,b\}$,
	\item $o = max \bigl( a - n + 1 , 0 \bigr)$,
	\item  $e \in \NN$, 
	\item $e_1 = \begin{cases}  0 & \text{if $p + q - 1 = a$} , \\ e & \text{if $p + q - 1 = b$},  \end{cases}$
	\item $e_2 = \begin{cases} 0 & \text{if $x + y = a$} , \\ e & \text{if $x + y = b$}  , \end{cases}$
	\item $m' = m - 2o - e$.
	\end{itemize}
There is a hamiltonian path~$\hampath$ from $(p,q)$ to $(x,y)$ in $\Bmn$, such that exactly~$e$ rowful inner subdiagonals travel east\/, if and only if 
	\begin{enumerate}
	\item \label{stretch-bound}
	$0 \le e \le \max \bigl( \min(m - n , b - a - 1) , 0 \bigr)$,
	\item \label{stretch-even}
	$e$~is even if $m + n$ is even,
	and
	\item \label{stretch-HP}
	there is a hamiltonian path~$\hampath'$ from $(p - o - e_1, q)$ to $(x - o - e_2,y)$ in $\board{m'}{n}$, such that either
		\begin{enumerate}
		\item \label{stretch-HP-north}
		all non-terminal diagonals travel north in~$\hampath'$,
			and $m' \ge n$,
		or
		\item \label{stretch-HP-east}
		all non-terminal diagonals travel east in~$\hampath'$,
			and
			$m' = n \ge a + 3$.
		\end{enumerate}
	\end{enumerate}
\end{prop}

\begin{proof}
We prove only ($\Rightarrow$), but the argument can be reversed. \refnote{StretchReverse}

Note, first, that $\max \bigl( \min(m - n , b - a - 1) , 0 \bigr)$ is the number of rowful inner subdiagonals, so \pref{stretch-bound} is obvious. In addition, if $m + n$ is even, then every diagonal is the union of two distinct subdiagonals (which means that the subdiagonals counted by~$e$ come in pairs), so $e$~must be even. This establishes \pref{stretch-even}.

By \fullcref{OuterCan}{East}, we may assume that all outer diagonals travel east. The definition of~$o$ implies that it is the number of rowful outer diagonals. Furthermore, for any outer diagonal $S_i \cup S_j$, we have $i < a \le b < j$ (and, by assumption, we have $p + q \in \{a + 1, b + 1\}$ and $x + y \in \{a,b\}$). Therefore
	$$ \Delta_{i,j}(p,q) = \Delta_{i,j}(x,y) = |\{i\}| = 1 .$$ 
Therefore, by repeated application of \cref{IgnoreEast}, we conclude that there is a hamiltonian path~$\hampath'$ from $(p - o,q)$ to $(x - o,y)$ in $\board{m - 2o}{n}$, such that exactly $e$ rowful inner subdiagonals travel east in~$\hampath'$.

The definitions of $e_1$ and~$e_2$ imply that 
	$$ \text{$e_1 = \sum \Delta_{i,j}(p,q)$ \quad and \quad $e_2 = \sum \Delta_{i,j}(x,y)$} ,$$
where the sums are over all rowful inner diagonals $S_i \cup S_j$ that travel east. Therefore, repeated application of \cref{IgnoreEast} to~$\hampath'$ yields a hamiltonian path $\hampath''$ from $(p - o - e_1, q)$ to $(x - o - e_2,y)$ in $\board{m - 2o - e}{n}$, such that no rowful inner diagonals travel east.
\noprelistbreak
\begin{itemize}
	\item If $a \le n-1$, then $o = 0$ and $e \le m - n$, so 
		$m' \ge  m - 2(0) - (m-n) = n$.
	\item If $a \ge n$, then $o = a - n + 1$ and $e \le \max(b - a - 1,0) \le b - a$, so
		\begin{align*}
		 m' &= m - 2o - e \ge m - 2(a - n + 1) - (b - a) 
			\\&= m  + 2n - 2 - (a + b)
			= n + 1 
			> n
		. \end{align*}
	\end{itemize}
In either case, we have $m' \ge n$.

The terminal diagonal of~$\hampath''$ is $S_{a'} \cup S_{b'}$, where
	$$  \begin{cases}
		a' = a - o & \text{if $x + y = a$} , \\
		b' = b - o - e & \text{if $x + y = b$}
		. \end{cases} $$
By definition, we have $o \ge a - n + 1$. Therefore:
	\begin{itemize}
	\item If $x + y = a$, then
	$$ a'
	= a - o
	\le a - (a - n + 1)
	= n - 1 ,$$
so $(0,n-1)$ is not in an outer diagonal of~$\hampath'$.
	\item If $x + y = b$, then
	\begin{align*}
	m' - 2
	&= m - 2o - e - 2
	= (m - o - e) - o - 2
	\\&\le  (m - o - e) - (a - n + 1) - 2
	= b'
	, \end{align*}
so $(m'-2,0)$ is not in an outer diagonal of~$\hampath'$.
	\end{itemize}
In either case, \fullcref{OuterCan}{North} tells us that changing all of the outer diagonals of~$\hampath''$ to travel north yields a hamiltonian path $\hampath'''$ with the same endpoints. If no inner diagonal travels east in~$\hampath''$, then all non-terminal diagonals travel north in~$\hampath'''$, so $\hampath'''$ is a hamiltonian path as described in conclusion~\pref{stretch-HP-north}.

We may now assume some inner diagonal $S_i \cup S_j$ travels east in~$\hampath''$. From the definition of~$\hampath''$, we know $S_i \cup S_j$ is not rowful. So \cref{CanGoEast} tells us that all inner diagonals travel east.
Since we have already assumed (near the start of the proof) that all outer diagonals travel east, this implies that all non-terminal diagonals travel east. 

Note that there are no rowful non-terminal diagonals in~$\board{m'}{n}$. (All inner diagonals travel east, but, by the definition of~$\hampath''$, no rowful diagonal travels east.)
Since, by the assumption of the preceding paragraph, 
 some inner diagonal travels east, this implies there must be at least one inner diagonal that is not rowful.
 So $a \le n-3$.

All that remains is to show that $m' = n$. Suppose not, which means $n < m'$. So $S_{n-1} \cup S_{m'-2}$ is a rowful diagonal, and $n-1 \le m' - 2$. This is not the terminal diagonal $S_a \cup S_b$, because $a \le n-3$. This contradicts the fact that there are no rowful non-terminal diagonals.
\end{proof}

\section{The general case} \label{GeneralSect}

In this \lcnamecref{GeneralSect}, we utilize \cref{stretch} and the results of \cref{AllNorth} to determine which pairs of squares in~$\Bmn$ are joined by a hamiltonian path~$\hampath$, and use this information to establish \cref{InitSquares,TermSquares}. First, \cref{InitIsTau+E} sharply restricts the possibilities for the initial square (perhaps after replacing $\hampath$ with its inverse). Then \cref{TermFor0q,TermForp0Small,TermForp0Large} determine the terminal squares of the hamiltonian paths (if any) that start at each of these potential initial squares. \Cref{InitSquares,TermSquares} are straightforward consequences of these much more detailed results.

\Cref{stretch} will be employed several times in this \lcnamecref{GeneralSect}. To facilitate this, we fix the following notation:

\begin{notation}
Given a hamiltonian path~$\hampath$ in~$\Bmn$ (with $m \ge n \ge 3$), we let:
	\begin{itemize}
	\item $S_a \cup S_b$ be the terminal diagonal of~$\hampath$, with $a \le b$,
	\item $(x,y)$ be the terminal square of~$\hampath$ (so $(x,y) \in S_a \cup S_b$),
	\item $(p,q)$ be the initial square of~$\hampath$ (so $p + q - 1 \in \{a,b\}$ by \cref{iotaE=tau}),
	\item $o = \max \bigl( a - n + 1 , 0 \bigr)$,
	\item $e$ be the number of rowful inner subdiagonals that travel east in~$\hampath$,
	\item $e_1 = \begin{cases}  0 & \text{if $p + q - 1 = a$} , \\ e & \text{if $p + q - 1 = b$},  \end{cases}$
	\item $e_2 = \begin{cases} 0 & \text{if $x + y = a$} , \\ e & \text{if $x + y = b$}  , \end{cases}$
	\item $m' = m - 2o - e$.
	\item $p' = p - o - e_1$, 
	\item $x' = x - o - e_2$,
	\item $\hampath'$ be the hamiltonian path from $(p', q)$ to $(x',y)$ in~$\board{m'}{n}$ that is provided by \cref{stretch},
	\item $S_{a'} \cup S_{b'}$ be the terminal diagonal of~$\hampath'$, with $a' \le b'$, so 
		$$ \text{$a' = a - o$ \ and \ $b' = b - o - e$} ,$$ 
	\item $\term_+$ be the southeasternmost square of~$S_b$ in~$\Bmn$,
	\item $\term_+'$ be the southeasternmost square of~$S_{b'}$ in~$\board{m'}{n}$.
	\end{itemize}
\end{notation}

\begin{prop} \label{InitIsTau+E}
Assume $\hampath$ is a hamiltonian path in~$\Bmn$ with $m \ge n \ge 3$.
Then the initial square of either $\hampath$ or the inverse of~$\hampath$ is $\term_+E$, unless all inner diagonals travel east, in which case, the initial square \textup(of either~$\hampath$ or~$\widetilde\hampath$\textup) might also be of the form $(p,0)$, with $1 \le p \le \lfloor (m + n)/2 \rfloor - 1$.
\end{prop}

\begin{proof}
We consider the two possibilities presented in \fullcref{stretch}{HP} as separate cases.

\setcounter{case}{0}

\begin{case}
Assume all non-terminal diagonals travel north in~$\hampath'$.
\end{case}
From the conclusion of \cref{Init=TermE}, we see that there are three possibilities to consider (perhaps after replacing $\hampath$ with its inverse, which also replaces $\hampath'$ with its inverse).

\begin{subcase}
Assume the initial square of $\hampath'$ is $\term_+' E$.
\end{subcase}
Then the initial square of $\hampath$ is $\term_+ E$.

\begin{subcase} \label{InitIsTau+EPf-transpose}
Assume $m' = n = a' + 2 = b' + 1$ and the initial square of the transpose of~$\hampath'$ is $\term_+' E$. 
\end{subcase}
Since $\term_+' = (b',0) = (m'-1,0)$, we have $\term_+'E = (0,n-1)$, so the initial square of~$\hampath'$ is the transpose of this, namely $(n-1,0)$. This means $p' = n-1$ and $q = 0$, so the initial square of~$\hampath$ is 
	$$ (p,q) = (p' + o + e_1, q) = (n - 1 + o + e_1, 0) ,$$
which is obviously of the form $(p,0)$. (Also, since $a' \neq b'$, we have $a \neq b$, so $a < (m + n - 3)/2$, which means $a \le \lfloor (m + n - 4)/2 \rfloor = \lfloor (m + n)/2 \rfloor - 2$. Therefore $p = p + q = a + 1 \le \lfloor (m + n)/2 \rfloor -  1$.) Furthermore, since $a' + 2 = b' + 1$, we have $a' + 1 = b'$, so $\hampath'$ has no inner diagonals. This implies that every inner diagonal of~$\hampath$ travels east.

\begin{subcase} \label{InitIsTau+EPf-n0}
Assume $a' + 1 = b' = n = m'-2$ and the initial square of $\hampath'$ is $(n,0)$.
\end{subcase}
Then $e_1 = 0$ (since $p' + q - 1 = n - 0 - 1 = a'$) and the initial square of~$\hampath$ is $( n + o, 0)$, which is of the form $(p,0)$.   (Also, since $a' \neq b'$, we have $p \le \lfloor (m + n)/2 \rfloor -  1$, as in the 
preceding 
\lcnamecref{InitIsTau+EPf-transpose}.)
Furthermore, since $a' + 1 = b'$, we know that $\hampath'$ has no inner diagonals, so every inner diagonal of~$\hampath$ travels east.

\begin{case}
Assume all non-terminal diagonals travel east in~$\hampath'$, and $m' = n \ge a' + 3$.
\end{case}
Since $m' = n$, we may let $(\hampath')^*$ be the transpose of~$\hampath'$. All non-terminal diagonals travel north in~$(\hampath')^*$ (and $n \notin \{a + 1, a + 2\}$), so \cref{Init=TermE} tells us (perhaps after replacing $\hampath$ with its inverse) that the initial square~$(\init')^*$ of $(\hampath')^*$ is $\term_+' E$. 
We have
	$$ b' = m' + n - (a' + 3) \ge m' + 0 = m' ,$$
so $\term_+' = (m'-1, b' - m' + 1)$, which means 
	$ \term_+' E 
	= (0, a' + 1 ) $.
Therefore the initial square of $\hampath'$ is the transpose of this, namely $(a' + 1, 0)$. So the initial square of~$\hampath$ is $(a' + 1 + o, 0) = (a + 1, 0)$, which is of the form $(p,0)$. (Also, since $a' \le n - 3$, we have $a' \neq b'$, so $p \le \lfloor (m + n)/2 \rfloor -  1$, as in the 
preceding 
\lcnamecref{InitIsTau+EPf-transpose}s.)
\end{proof}

It is important to note that the possibilities for the square $\term_+E$ can be described quite precisely:

\begin{lem} \label{tau+E}
If $\hampath$ is a hamiltonian path in~$\Bmn$, with $m \ge n$, then either: \refnotelower{TauE}
	\begin{enumerate}
	\item $\term_+ E = (0,q)$ with $1 \le q \le n-1$,
	or
	\item $\term_+ E = (p, 0)$, with $\lceil (m + n - 1)/2 \rceil \le p \le m-1$ \textup(and $m \neq n$\textup).
	\end{enumerate}
\end{lem}

Thus, \cref{InitIsTau+E} tells us that the initial square of~$\hampath$ is of the form $(0,q)$ or $(p,0)$ (perhaps after passing to the inverse). We will find all of the corresponding terminal squares in \cref{TermFor0q,TermForp0Small,TermForp0Large}.

\begin{prop} \label{TermFor0q}
Assume $\init =  (0,q)$ with $1 \le q \le n-1$ and $m \ge n \ge 3$. There is a hamiltonian path~$\hampath$ in~$\Bmn$ from $\init$ to $(x,y)$ if and only if either
	\begin{enumerate}
	
	\item \label{TermFor0q-a}
	$x + y = q - 1 \ge \nm$ and $\nm \le x \le \mm$,
	or
	
	\item \label{TermFor0q-b}
	$x + y = m + n - q - 2$ and $\mf \le x \le m - \np$ \textup(and $q \ge \nm$\textup).
	
	\end{enumerate}
\end{prop}

\begin{proof}
We prove only ($\Rightarrow$), but the argument can be reversed. \refnote{TermFor0qReverse}
Note that $a = a' = q - 1$, 
	$$ b = m + n - 3 - a = m + n - q - 2 ,$$
$o = e_1 = 0$, and $\term_+' = (m' - 1, n - q - 1)$.
Also note that, since $a \le n-2$, the largest possible value of~$e$ is $m - n$. \refnote{LargestE}
\fullCref{stretch}{HP} gives us two cases to consider.

\setcounter{case}{0}

\begin{case}
Assume all non-terminal diagonals travel north in~$\hampath'$.
\end{case}
\Cref{UsualEndpt} tells us there are (at most) two possibilities for the terminal square $(x',y)$.

\begin{subcase}
Assume $x' + y = a$ and $x' = \lfloor (m' - 1)/2 \rfloor$.
\end{subcase}
We have $x + y = a = q - 1$ and (since $o = e_2 = 0$)
		$$x 
		= x'
		= \lfloor (m' - 1)/2 \rfloor
		= \lfloor (m - e - 1)/2 \rfloor
		.$$
The smallest possible value of~$e$ is~$0$, so $x \le \mm$. 
Conversely, since the largest possible value of~$e$ is $m - n$,
we have $x \ge \lfloor (m - (m-n) - 1)/2 \rfloor = \nm$. (Therefore $q - 1 \ge x \ge \nm$.)

\begin{subcase}
Assume $x' + y = b' \neq a$ and $x' = \lfloor m'/2 \rfloor$.
\end{subcase}
We have 
	$$x + y = b = m + n - q - 2$$
and 
		$x 
		= x' + e
		= \lfloor (m + e)/2 \rfloor$. 
The smallest possible value of~$e$ is~$0$, so $x \ge \mf$.
Conversely, since the largest possible value of~$e$ is $m - n$, we have
	$$x \le \lfloor (2m - n)/2 \rfloor = m - \lceil n/2 \rceil = m - \np .$$
	Therefore 
		$$m + n - q - 2 
		= x + y
		\le (m - \np) + (n - 1) 
		,$$
	so $q \ge \np - 1 = \nm$.

\begin{case}
Assume all non-terminal diagonals travel east in~$\hampath'$ and $m' = n \ge a + 3$.
\end{case}
Since $m' = n$, we may let $(\hampath')^*$ be the transpose of~$\hampath'$. Then all non-terminal diagonals travel north in~$(\hampath')^*$ (and $n \notin \{a+1,a+2\}$), so \cref{Init=TermE} tells us that either the initial square or the terminal square of~$(\hampath')^*$ is
	$\term_+'E = (0, q)$. 
Since the initial square of~$(\hampath')^*$ is $(q,0) \neq (0, q-1)$, we conclude that the terminal square is $(0,q)$, so $q \in \{a,b'\}$. However, $q = a + 1$ and $b' = m' + n - 3 - a \ge m' > q$. This is a contradiction.
\end{proof}

\begin{prop} \label{TermForp0Small}
Assume $1 \le p \le \lfloor (m + n)/2 \rfloor - 1$ and $m \ge n \ge 3$. There is a hamiltonian path~$\hampath$ in~$\Bmn$ from $(p,0)$ to $(x,y)$, such that all inner diagonals travel east, if and only if either
	\begin{enumerate}
	
	\item \label{TermForp0Small-y=nm}
	$x + y = p - 1 \ge \nm$,  and $y = \nm$,
	or
	
	\item \label{TermForp0Small-y=nf}
	$x + y = m + n - p - 2$, $y = \nf$, and $\nm \le p \le n-1$.

	\end{enumerate}
\end{prop}

\begin{proof}
We prove only ($\Rightarrow$), but the argument can be reversed. \refnote{TermForp0SmallReverse}
We may assume that all non-terminal diagonals travel east in~$\hampath$ \fullcsee{OuterCan}{East}.
Note that $q = 0$, $a = p-1$, and $e_1 = 0$.
We consider two cases. 

\setcounter{case}{0}

\begin{case}
Assume $p \le n - 1$.
\end{case}
Then $\term_+ = (m-1,n - p - 1)$ and $\term_+E = (0,p)$.
We have $a = p - 1 \le n - 2$, so $e = m - n$ and $o = 0$, so 
	$$m' = m - e = m - (m -n) = n .$$
Therefore, we may let $(\hampath')^*$ be the transpose of~$\hampath'$. The initial square of~$(\hampath')^*$ is $(0,p) = \term_+'E$, and all non-terminal diagonals travel north in~$(\hampath')^*$, so \cref{UsualEndpt} tells us  there are only two possible terminal squares. (Note that, since we are considering the transpose, the role of~$x$ in \cref{UsualEndpt} is played by~$y$ here.)

\begin{subcase}
Assume $x' + y = a$ and $y = \lfloor (m' - 1)/2 \rfloor$.
\end{subcase}
We have
	$x + y = x' + y = a = p - 1$ and 
		$y
		= \nm$.
	Since $x + y = p - 1$, this implies $p - 1 \ge \nm$.
	
\begin{subcase}
Assume $x' + y = b'$ and $y = \lfloor m'/2 \rfloor$.
\end{subcase}
We have 
	$$x + y = b = m + n - 3 - a = m + n - p - 2$$
and
		$y
		= \nf$.
		Since $x + y = m + n - p - 2$ and $x \le m - 1$, this implies $p \ge \nm$.\refnote{TermForp0SmallPf-nm}

\begin{case}
Assume $p \ge n$.
\end{case}
All inner diagonals are rowful (since $a = p - 1 \ge n-1$), and they all travel east in~$\hampath$ (by assumption), so $\hampath'$ has no inner diagonals, which means $b' \le a' + 1$. Also, since 
	$$a = p - 1 \le \left\lfloor \frac{m + n - 4}{2} \right\rfloor < \frac{m + n - 3}{2} \le b,$$
we know $a \neq b$, so $a' \neq b'$. Therefore $b' = a' + 1$.

Since $a = p-1 \geq n - 1$, we have $o = a - n + 1 = p - n$, which means $p' = n$, so $a' = n - 1$. Since $b' = a' + 1$, this implies $m' = n + 2$. \refnote{TermForp0SmallPf-m'}
Then $m' \neq n$, so \fullcref{stretch}{HP} tells us that all non-terminal diagonals travel north in~$\hampath'$.
Then \cref{init=tauplus} tells us that $y = \lfloor (m' - 1)/2 \rfloor - 1 = \nm$. We also have $x + y = a = p - 1$ (and $p - 1 \ge n - 1 \ge \nm$).
\end{proof}

\begin{prop} \label{TermForp0Large}
Assume $\lfloor (m + n)/2 \rfloor \le p \le m-1$, and $m > n \ge 3$. There is a hamiltonian path~$\hampath$ in~$\Bmn$ from $(p,0)$ to $(x,y)$ if and only if either
	\begin{enumerate}
	
	\item \label{TermForp0Large-a}
	$x + y = m + n - p - 2$, $m - p + \nm \le x \le \mm$, and $p \neq (m + n - 1)/2$,
	or
	
	\item \label{TermForp0Large-b}
	$x + y = p - 1$ and $\mf \le x \le p - \np$,
	or
	
	\item \label{TermForp0Large-special}
	$x = \mm$, $y = \nm$, and $p = (m + n - 1)/2$ \textup(so $m + n$ is odd\/\textup).

	\end{enumerate}
\end{prop}

\begin{proof}
We prove only ($\Rightarrow$), but the argument can be reversed. \refnote{TermForp0LargeReverse}
Note that $b = p-1$, $o = a - n + 1$ (since $b \le m - 2$), $a' = n - 1$, $q = 0$, $e_1 = e$, 
	and
	$\term_+' = (b',0)$.
The largest possible value of~$e$ is $2p - m - n$, \refnote{TermForp0LargePf-e}
unless $m + n$~is odd and $p = (m + n - 1)/2$, in which case the only value of~$e$ is~$0$.
As usual, \fullCref{stretch}{HP} gives us two cases to consider.

\setcounter{case}{0}

\begin{case}
Assume all non-terminal diagonals travel north in~$\hampath'$.
\end{case}
\Cref{UsualEndpt} tells us there are (at most) two possibilities for the terminal square $(x',y)$ of~$\hampath'$.

\begin{subcase}
Assume $x' + y = a'$ and $x' = \lfloor (m' - 1)/2 \rfloor$.
\end{subcase}
We have $x + y = a = m + n - 3 - b = m + n - p - 2$. Also,
	$$x = x' + o + e_2 
	= \lfloor (m' - 1)/2 \rfloor + o + 0
	= \lfloor (m - e - 1)/2 \rfloor
	. $$
The smallest possible value of~$e$ is $0$, so $x \le \mm$.
Conversely, if $p \neq (m + n - 1)/2$, then the largest possible value of~$e$ is $2p - m - n$,
so $x \ge m - p + \nm$. 
However, if $p = (m + n - 1)/2$, then $e = 0$, so $x = \mm$ and 
	$$y = a - x 
	= m + n - p - 2 - \mm
	= \nm 
	.$$

\begin{subcase}
Assume $x' + y = b'$, $x' = \lfloor m'/2 \rfloor$, and $a' \neq b'$.
\end{subcase}
We have $x + y = b = p - 1$. Also, since $a' \neq b'$, we have $a \neq b$, so $b \neq (m + n - 3)/2$, which means $p \neq (m + n - 1)/2$.
Furthermore,
	$$x 
	= x' + o + e
	= \lfloor m' /2 \rfloor + o + e
	= \lfloor (m + e)/2 \rfloor
	. $$
The smallest possible value of~$e$ is $0$, so $x \ge \mf$.
	Conversely, since the largest possible value of~$e$ is $2p - m - n$,
	we have $x \le p - \np$. 

\begin{case}
Assume all non-terminal diagonals travel east in~$\hampath'$, and $m' = n \ge a' + 3$.
\end{case}
Since $m' = n$, we may let $(\hampath')^*$ be the transpose of~$\hampath'$. Then all non-terminal diagonals travel north in~$(\hampath')^*$, so \cref{Init=TermE} tells us that $\term_+'E = (p',0)$ is either the initial square or the terminal square of~$(\hampath')^*$. Since the initial square of~$(\hampath')^*$ is $(0,p') \neq (p',0)$, it must be the terminal square that is $(p',0)$. So the inverse of~$(\hampath')^*$ is a hamiltonian path from $(p',0) = \term_+'E$ to $(0,p')$. This contradicts \cref{UsualEndpt}.
\end{proof}

\begin{proof}[\bf Proof of \cref{InitSquares,TermSquares}]
It is immediate from \cref{InverseHP} that a square $\sigma$ is the initial square of a hamiltonian path if and only if its inverse $\widetilde\sigma$ is the terminal square of a hamiltonian path. Therefore, \cref{InitSquares,TermSquares} are logically equivalent: the squares listed in one \lcnamecref{InitSquares} are simply the inverses of the squares listed in the other. So it suffices to prove \cref{TermSquares}. That is, we wish to show that the terminal squares (and the inverses of the initial squares) listed in \cref{TermFor0q,TermForp0Small,TermForp0Large} combine to give precisely the squares listed in \cref{TermSquares}.

\setcounter{case}{0}

\begin{case}
The inverses of the initial squares in \cref{TermFor0q}.
\end{case}
The set of initial squares is $\{\, (0,q) \mid \nm \le q \le n - 1 \,\}$. Their inverses form the set
	$\{\, (m-1, y) \mid 0 \le y \le n - 1 - \nm \,\}$.
Since $n - 1 - \nm = \nf$, these are precisely the squares in \fullcref{TermSquares}{m}.

\begin{case}
The inverses of the initial squares in \cref{TermForp0Small,TermForp0Large}.
\end{case}
The combined set of these initial squares is $\{\, (p,0) \mid \nm \le p \le m - 1 \,\}$. Their inverses form the set
	$\{\, (x, n-1) \mid 0 \le x \le m - 1 - \nm \,\}$.
Since $m - 1 - \nm = m - \np$, these are precisely the squares in \fullcref{TermSquares}{n}.

\begin{case}
The terminal squares in \cref{TermForp0Small}, and also~\fullref{TermForp0Large}{special} when $m$~is even and $n$~is odd.
\end{case}
The terminal squares in \fullcref{TermForp0Small}{y=nm} have $y = \nm$, so $x = p - 1 - \nm = p - \np$. Since $p$ can take on any value from $\np$ to $\lfloor (m + n)/2 \rfloor - 1$, this means that $x$ ranges from~$0$ to \refnotelower{TermForp0Small-y=nm-UpperBound}
	$$ \lfloor (m + n)/2 \rfloor - 1 - \np =
		\begin{cases}
		\mf - 2 & \text{if $m$ is even and $n$~is odd} , \\
		\mf - 1 & \text{otherwise}
		. \end{cases} $$
Thus, these are precisely the squares listed in \fullcref{TermSquares}{nm}, except that the square $(\mf - 1, \nm)$ is missing when $m$~is even and $n$~is odd. Fortunately, in this case, the missing square is precisely the square listed in \fullcref{TermForp0Large}{special}.

The terminal squares in \fullcref{TermForp0Small}{y=nf} have $y = \nf$, so $x$ ranges from
	$$ m + n - (n-1) - 2 - \nf
	= m - \nf - 1 $$
to
	$$   m + n - \nm - 2 - \nf = m - 1 .$$
Thus, these are precisely the squares listed in \fullcref{TermSquares}{nf}.

\begin{case} \label{ListBottom}
The terminal squares in \namecref{TermFor0q}s \fullref{TermFor0q}{a} and \fullref{TermForp0Large}{a}, and also \fullref{TermForp0Large}{special} when $m$~is odd and $n$~is even.
\end{case}
The terminal squares in \fullcref{TermFor0q}{a} are:
	$$ \{\, (x,y) \mid \nm \le x \le \mm, \ x + y \le n - 2  \,\} .$$
Since $x \ge \nm$, we have $y \le \nf - 1$. So these are the squares listed in \fullcref{TermSquares}{bottom} that satisfy $x + y \le n - 2$.

Now consider \fullcref{TermForp0Large}{a}. 
Since $x + y = m + n - p - 2$, the constraint $x \ge m - p + \nm$, can be replaced with 
	$$y \le m + n - p - 2 - (m - p + \nm) = n - \nm - 2 = \nf - 1 .$$
Also, the range $\lceil (m + n)/2 \rceil \le p \le m-1$ means that $x + y$ is allowed to take any value from $n - 1$ to 
	$$m + n - \lceil (m + n)/2 \rceil - 2 = \lfloor (m + n)/2 \rfloor - 2 .$$
Since $x + y \ge n - 1$ and $y \le \nf - 1$, we have $x \ge (n-1)-(\nf-1) > \nm$.
Thus, the  terminal squares in \fullcref{TermFor0q}{b} are:
	$$ \{\, (x,y) \mid \nm \le x \le \mm, \  y \le \nf - 1, \, n-1 \le x + y \le \lfloor (m + n)/2 \rfloor - 2  \,\} .$$

Therefore, the union of these two sets consists of precisely the squares $(x,y)$ listed in \fullcref{TermSquares}{bottom} that satisfy 
	$$ x + y \le \lfloor (m + n)/2 \rfloor - 2 .$$
However, any square $(x,y)$ listed in  \fullcref{TermSquares}{bottom} satisfies
	$ x + y \le \mm + \nf - 1$.
Since
	$$ \lfloor (m + n)/2 \rfloor - 2 
	\ge \frac{m}{2} + \frac{n}{2} - \frac{5}{2} 
	\ge \mm + (\nf - 1) - 1 ,$$
we conclude that in order for a square $(x,y)$ of \fullcref{TermSquares}{bottom} to be missing from the union, equality must hold throughout (so $m$~is odd and $n$~is even) and we must have $x = \mm$ and $y = \nf - 1 = \nm$. This is precisely the square listed in \fullcref{TermForp0Large}{special}. So these three sets together constitute the squares listed in \fullcref{TermSquares}{bottom}.

\begin{case}
The terminal squares in \fullcref{TermFor0q}{b} and \fullcref{TermForp0Large}{b}.
\end{case}
First, we consider \fullcref{TermForp0Large}{b}. Since $x + y = p - 1$, the constraint $x \le p - \np$ can be replaced with
	$y \ge \nm$.
Also, since $x + y \le m - 2$, this implies $x \le m - 2 - \nm < m - \np$.
Therefore, the set of terminal squares is 
	$$ \{\, (x,y) \mid \mf \le x \le m - \np, \, \nm \le y \le n - 1, \ \lfloor (m + n)/2 \rfloor - 1 \le x + y \le m - 2 \,\} .$$

We now consider \fullcref{TermForp0Large}{b}.
The constraint $\nm \le q \le n - 1$ means that $x + y$ is allowed to take any value from $m - 1$ to $m + \nf - 1$. However, since $x \le m - \np$ and $y \le n - 1$, the upper bound is redundant. 
Also, since $x \le m - \np$ and $x + y \ge m - 1$, we must have $y \ge \nm$.
Therefore, the set of terminal squares is 
	$$ \{\, (x,y) \mid \mf \le x \le m - \np, \, \nm \le y \le n - 1, \ m - 1 \le x + y \,\} .$$

Therefore, the union of these two sets consists of precisely the squares $(x,y)$ listed in \fullcref{TermSquares}{top} that satisfy 
	$$ x + y \ge \lfloor (m + n)/2 \rfloor - 1 .$$
However, it is easy to see that every one of the squares satisfies this condition \refnote{TermSquaresTopLowerBound}
(since $x \ge \mf$ and $y \ge \nm$), so we conclude that these two sets constitute the squares listed in \fullcref{TermSquares}{top}.
\end{proof}

\begin{rem} \label{n=1or2}
The above results assume $n \ge 3$. For completeness, we state, without proof, the analogous results for $n = 1,2$. It was already pointed out in \cref{n=2HC} that every square is both an initial square and a terminal square in these cases, but we now provide a precise list of the pairs of squares that can be joined by a hamiltonian path.

	\begin{enumerate}
	
	\item Assume $n = 1$. Then every square in the board is of the form $(*,0)$. There is a hamiltonian path from $(p,0)$ to $(x,0)$ if and only if $(p,0) = (x,0)E$. \refnote{n=1Pf}
More precisely, every hamiltonian path is obtained by removing an edge from the hamiltonian cycle $(E^m)$.

\item  \label{n=1or2-2}
Assume $m \geq n = 2$. \Cref{Endpts-n=2Fig} 
lists the initial square $(p,q)$ and terminal square $(x,y)$ of every hamiltonian path in $\board{m}{2}$. \refnote{n=2Pf}
	\end{enumerate}
\end{rem}

\begin{figure}[b]
\renewcommand{\arraystretch}{1.2} 
$$\begin{array}{|c|c|c|c|}
\hline
		& \hbox{initial} & \hbox{terminal} & \\[-5pt] 
		& \hbox{square} & \hbox{square} & \hbox{restrictions, if any} \\[-5pt] 
		& (p,q)		& (x,y) 		& \\
\hline\hline
A_2 & (p,q) & (p,q)E^{-1} &  \\
B_2 & (0,1) & (0,0) & \\
C_2 & (1,0) & (m-2,1) & \\
D_2 & (m-1,1) & (m-1,0) & \\
E_2 & (0,1) & (m-2,1) & m \ge 3 \\
F_2 & (1,0) & (m-1,0) & \text{$m$ is odd} \\
G_2 & (1,0) & (m-1,0) & m \ge 4 \\
H_2 & (p,1) & (m-2-p,1) & \text{$m \ge 4$ \ and \ $\mp \le p \le m-2$} \\
I_2 & (p,0) & (m-p,0) & \text{$m \ge 4$ \ and \ $p \ge \mp + 1$} \\
J_2 & (p,0) & (p-2,1) & \text{$m \ge 5$ \ and \ $p \notin \{ 1, \mf, \mf+1 \} \smallsetminus \{\mm\}$} \\
\hline
\end{array}$$
\caption{Endpoints of the hamiltonian paths in $\Bmn$ when $m \ge n = 2$.}
\label{Endpts-n=2Fig}
\end{figure}

\nocite{*}
\bibliographystyle{amsplain}

\bibliography{}


\AtEndDocument{

\newpage

\thispagestyle{empty}

\markboth{Appendix: Notes to aid the referee}{Appendix: Notes to aid the referee}

\begin{appendix}

\section{Notes to aid the referee}

\bigskip\hrule width\textwidth \bigbreak

\begin{aid} \label{OuterCanPf}
\textbf{Proof of \cref{OuterCan}.}
Each of~$\hampath_E$ and~$\hampath_N$ is the union of a path from~$\init$ to~$\term$ and a (possibly empty) collection of disjoint cycles. (In the terminology of \cite[Defn.~2.6]{Forbush}, $\hampath_E$ and~$\hampath_N$ are \emph{spanning quasi-paths}.) Let $S_a \cup S_b$ be the terminal diagonal (with $a \le b$), and let $W$ be the union of $S_a \cup S_b$ and all of the inner diagonals. Since $\hampath$ is connected, \cref{OuterEastOK} implies that all of~$W$ is contained in a single component of~$\hampath_E$, and also in a single component of~$\hampath_N$. 

\pref{OuterCan-East}
Let $\sigma$ be any square of~$\Bmn$ that not in~$W$, so $\sigma$ is on some outer subdiagonal~$S_i$. If $i < a$, then $\sigma E^{a-i} \in S_a$; if $i > b$, then $\sigma \in S_b E^{i-b}$. In either case, we see that $\sigma$ is in the same component of~$\hampath_E$ as~$W$. So all of~$\Bmn$ is in a single component of~$\hampath_E$, which means that $\hampath_E$ is a hamiltonian path.

\pref{OuterCan-North}
If $S_{n-1} \cup S_{m-2}$ is an outer diagonal, then 
	$$ \text{$\bigl\{\, (x,y) \in \Bmn \mid x \in \{0,m-1\} \,\bigr\}$ is disjoint from~$W$} ,$$
and therefore travels north in~$\hampath_N$. This means that $\hampath_N$ contains the cycle $[(0,0)](N^{2n})$, and therefore is not a hamiltonian path.

Now suppose $S_{n-1} \cup S_{m-2}$ is not an outer diagonal, which means that it is contained in~$W$ (and $a \le \min(n-1,m-2)\le \max(n-1,m-2) \le b$).
Let $(x,y)$ be any square of~$\Bmn$ that not in~$W$, so $(x,y)$ is on some outer subdiagonal~$S_i$. 
	\begin{itemize}
	\item If $i < a$, then $(x,y) N^{a-i} \in S_a$.
	\item If $i > b$, and $x \le b$, then $(x,y) \in S_b N^{i-b}$.
	\item If $x > b$, then $x > b \ge m - 2$, so we must have $b = m - 2$ and $x = m - 1$. Since $b = m - 2$, we see that every square in the eastmost column of~$\Bmn$ is on an outer diagonal, and therefore travels north in~$\hampath_N$. We also have $a = m + n - 3 - (m - 2) = n - 1$. So $(x,y) = (0,n-1)N^{y+1} \in S_a N^{y+1}$.
	\end{itemize}
 In each case, we see that $(x,y)$ is in the same component of~$\hampath_N$ as~$W$. So all of~$\Bmn$ is in a single component of~$\hampath_N$, which means that $\hampath_N$ is a hamiltonian path.
\end{aid}

\begin{aid} \label{Sa+1}
Since $\term_+E = (0, n - 1 - y)$, we have $\term_+EN^{-1} = (0, n - 2 - y)$. (Note that $n - 2 - y \ge 0$, since $y \neq n-1$ by \cref{00NotInit}.) So $\term_+EN^{-1} \in S_i$ and $\term_+E \in S_{i+1}$, where $i = n - 2 - y$. Since $\term_+EN^{-1}$ and~$\term_+E$ are both in the terminal diagonal $S_a \cup S_b$ (and $a \le b$), we conclude that $i = a$ and $b = i + 1 = a + 1$.
\end{aid}

\begin{aid} \label{ApplyToTranspose}
The arguments of \cref{Init=TermEPf=x=m-1} up to this point yield a contradiction unless $(m-1,0)$ and $(0,m-2)$ both travel north. By applying the same arguments to the transpose~$\widetilde\hampath$ of~$\hampath$, we conclude that these two squares must also travel north in~$\widetilde\hampath$. This means that the transposes of these squares, namely,  $(0,m-1)$ and $(m-2,0)$, travel east in~$\hampath$.
\end{aid}

\begin{aid} \label{TauPlusMustNorth}
Since $\term_+ E$ is the initial square, it has no in-arcs, so $\term_+$ cannot travel east. And $\term_+$ cannot be the terminal square. (Otherwise, adding the arc from $\term_+$ to $\term_+ E$ would yield a hamiltonian cycle in~$\Bmn$, contradicting \cref{00NotInit}.) So $\term_+$ must travel north.
\end{aid}

\begin{aid} \label{sigmaBisSquare}
We know that either $\sigma_a$ or~$\sigma_b$ is a square that travels east, so, in particular, either $\sigma_a$ or~$\sigma_b$ must exist as a square of~$\Bmn$. This means that either $a - \mm \ge 0$ or $b - \mm \le n - 1$. 

If $a - \mm \ge 0$, then
	\begin{align*}
	 b - \mm 
	 &= (m + n - 3 - a) - \mm 
	 = \mm + n - 2 - a 
	 \\&= n - 2 - (a - \mm) \le n - 2 
	 < n - 1 
	 . \end{align*}
Thus, in either case, we have $b - \mm \le n - 1$, so $\sigma_b$ exists. (That is, if $\sigma_a$ exists, then $\sigma_b$ also exists, but not conversely.)
\end{aid}

\begin{aid} \label{sigmaAGoesEast}
We assume the notation of \fullcref{circ-orient}{init}, and let $\sigma_a^- = (\mf - 1, a - \mf + 1)$. Since $\sigma_b$ travels east, we know from \fullcref{circ-orient}{init} that $v(\sigma_b) < v(\term)$. So the desired conclusion follows from the observation that $v (\sigma_a^-) < v(\sigma_b)$ (because this implies $v(\sigma_a^-) < v(\term)$, so we can apply \fullcref{circ-orient}{init}).  Indeed, since $\iota E^{-1} = \term_+$, we see that if $\term_-$ is the northwesternmost square on~$S_a$, then $v(\term_-) = 1$. Therefore:
	\begin{itemize}
	\item If $a = b$, then $\sigma_a = \sigma_b$, so $v(\sigma_a^-) = v( \sigma_a N E^{-1} ) = v( \sigma_b N E^{-1} ) = v(\sigma_b) - 1$.
	\item If $a < b$, then $v(\sigma) < v(\sigma')$ for all $\sigma \in S_a$ and $\sigma' \in S_b$.
	\end{itemize}
\end{aid}

\begin{aid} \label{AEastBNorth}
We assume the notation of \fullcref{circ-orient}{init}, and let $\sigma_a^- = (\mf - 1, a - \mf + 1)$ and $\sigma_b^+ = (\mf + 1, b - \mf - 1)$. Since $\sigma_a$ travels east, we know from \fullcref{circ-orient}{init} that $v(\sigma_a) < v(\term)$. 

Since $\sigma_a$ travels east, we know from \fullcref{circ-orient}{init} that $v(\sigma_a) < v(\term)$. Therefore 
	$$v (\sigma_a^-) = v( \sigma_a N E^{-1} ) = v(\sigma_a) - 1 < v(\term) - 1 < v(\term) ,$$
so \fullcref{circ-orient}{init} tells us that $\sigma_a^-$ travels east.

Since $\sigma_b$ travels north, we know from \fullcref{circ-orient}{init} that $v(\sigma_b) > v(\term)$. Therefore 
	$$v (\sigma_b^+) = v( \sigma_b E N^{-1} ) = v(\sigma_b) + 1 > v(\term) + 1 > v(\term) ,$$
so \fullcref{circ-orient}{init} tells us that $\sigma_b^+$ travels north.
\end{aid}

\begin{aid} \label{TermDiagAndCycle}
Suppose $(x,y)$ is a square that is in both the terminal diagonal and the cycle $[(\mf,0)](N^{2n})$.
Since $(x,y)$ is in the terminal diagonal, we have $(x,y) \in S_a \cup S_b$.
	\begin{itemize}

	\item If $(x,y) \in S_a$, then $x + y = a$. But we also have $x \ge \mf$, since $(x,y)$ is in the cycle. Therefore
	\begin{align*}
	 0 &\le y = a - x \le a - \mm = (m + n - 3 - b) - \mm 
	\\&= m + n - 3 - (n - 1 + \mf) - \mm 
	= -1 < 0 
	. \end{align*}
	This is a contradiction.
	
	\item If $(x,y) \in S_b$, then $x + y = b$, so 
		$$ n - 1 \ge y = b - x = (n - 1 + \mf) - x .$$
	Therefore $x \ge \mf$. However, since $(x,y)$ is in the cycle, we also have $x \le \mf$. Therefore $x = \mf$ (and $(x,y) \in S_b$), so $(x,y) = \sigma_b$.
	
	\end{itemize}
\end{aid}

\begin{aid} \label{SigmaBAlsoNorth}
We assume the notation of \fullcref{circ-orient}{init}. Since $\sigma_a$ travels north, we know from \fullcref{circ-orient}{init} that $v(\sigma_a) > v(\term)$. Furthermore, since $\init E^{-1} = \term_+$, it is clear that $v(\sigma) < v(\sigma')$ for all $\sigma \in S_a$ and $\sigma' \in S_b$, so, in particular, $v(\sigma_a) < v(\sigma_b)$. Therefore $v(\sigma_b) > v(\term)$, so \fullcref{circ-orient}{init} tells us that $\sigma_b$ travels north.
\end{aid}

%

\begin{aid} \label{StretchInitialSquare}
Perhaps it is not obvious that $\tau_0 E$ is the initial square of~$\hampath$ if and only if $\bigl( \Delta_{i,j}^\square(\tau_0) \bigr) E$ is the initial square of~$\hampath'$. However, if $\tau_0 E$ is the initial square of~$\hampath$, then $\tau_0$ does not travel east, and $\tau_0 EN^{-1}$ does not travel north. So $\Delta_{i,j}^\square(\tau_0)$ does not travel east, and $\Delta_{i,j}^\square(\tau_0)EN^{-1} = \Delta_{i,j}^\square(\tau_0 EN^{-1})$ does not travel north. So the initial square of~$\hampath'$ is $\Delta_{i,j}^\square(\tau_0) E$. And conversely.
\end{aid}

\begin{aid} \label{LotsEast}
We assume the notation of \fullcref{circ-orient}{init}. 
Since $\sigma_b$ travels east, we know from \fullcref{circ-orient}{init} that $v(\sigma_b) < v(\term)$. 

Since $\init E^{-1} = \term_+$, it is clear that $v(\sigma) < v(\sigma')$ for all $\sigma \in S_a$ and $\sigma' \in S_b$. Therefore $v(\sigma) < v(\sigma_b) < v(\term)$, so \fullcref{circ-orient}{init} tells us that $\sigma$ travels north for every $\sigma \in S_a$.

Also, since $(\mf - 1, b - \mf + 1) = \sigma_b N E^{-1}$, we have $v \bigl( (\mf - 1, b - \mf + 1) \bigr) = v(\sigma_b) - 1 < v(\sigma_b) < v(\term)$, so \fullcref{circ-orient}{init} tells us that $(\mf - 1, b - \mf + 1)$ travels north.
\end{aid}

\begin{aid} \label{tauOnSa}
We assume the notation of \fullcref{circ-orient}{init}. Note that $\init E^{-1} = (m-2,0)E^{-1} = (m-3,0)$ is the southeasternmost square on~$S_a$, so 
	\begin{align*}
	S_a 
	&= \{\,(m-3,0)(NE^{-1})^i \mid 0 \le i \le a \,\} 
	\\&= \{\,(m-3,0)(EN^{-1})^j \mid n - a \le j \le n \,\} 
	\\&= \{\, \sigma \in S_a \cup S_b \mid v(\sigma) \ge n - a \,\}
	. \end{align*}

The northwesternmost square on $S_a$ is $(0,a) = (0,n - 1)$, which does not travel north, so \fullcref{circ-orient}{init} tells us that $v(\term) \ge v \bigl( (0,a) \bigr) = n - a$. Hence, we have $\term \in S_a$.

For all $\sigma \in S_b$, we have $\sigma \notin S_a$, so $v(\sigma) < n - a \le v(\term)$, so \fullcref{circ-orient}{init} tells us that $\sigma$ travels east.
\end{aid}

\begin{aid} \label{AlsoNorth}
We assume the notation of \fullcref{circ-orient}{init}. 
Since $(\nf,\nf)$ travels north, \fullcref{circ-orient}{init} tells us that $v \bigl( (\nf,\nf) \bigr) > v(\term)$.
Then, since $(\nf + 1,\nf - 1) = (\nf,\nf)E N^{-1}$, we have 
	$$v \bigl( (\nf + 1,\nf - 1) \bigr) = v \bigl( (\nf,\nf) \bigr) + 1 >  v \bigl( (\nf,\nf) \bigr) > v(\term) .$$
So \fullcref{circ-orient}{init} tells us that $(\nf + 1,\nf - 1)$ travels north.
\end{aid}

\begin{aid} \label{CantEast}
We assume the notation of \fullcref{circ-orient}{init}. 
Suppose $(\nf,\nm)$ travels east. Then \fullcref{circ-orient}{init} tells us that $v \bigl( (\nf,\nm) \bigr) < v(\term)$.
Also, since $(\nf,\nm) = (\nm,\nf)E N^{-1}$, we have $v \bigl( (\nf,\nm) \bigr) = v \bigl( (\nm,\nf) \bigr) + 1$. 
Therefore
	$$v \bigl( (\nm,\nf) \bigr) = v \bigl( (\nf,\nm) \bigr) - 1 <  v \bigl( (\nf,\nm) \bigr) < v(\term) .$$
So \fullcref{circ-orient}{init} tells us that $(\nm,\nf)$ travels east. 

Furthermore, 
	$$ S_b
	= \{\, (x,  n - x) \mid 1 \le x \le n \,\}
	= \{\, \init E^{-1}(EN^{-1})^k \mid 1 \le k \le n \,\} ,$$
where $\init = (m-2,0)$ is the initial square. 
Therefore $v(\sigma_b) < v(\sigma_a)$ for all $\sigma_b \in S_b$ and $\sigma_a \in S_a$. Since 
	$$ \text{$v \bigl( (\nf,\nm) \bigr) < v(\term)$,
	\ and \ 
	$(\nf,\nm) \in S_a$} ,$$
we conclude that $v(\sigma_b) < v(\term)$, so \fullcref{circ-orient}{init} tells us that every square in~$S_b$ travels east. In particular, $(\nf,\nf)$ and $(\nm, \nf+1)$ do not all travel east.

This contradicts the fact that $(\nf,\nm)$, $(\nf,\nf)$, $(\nm,\nf)$, and $(\nm, \nf+1)$ do not all travel east.
\end{aid}

\begin{aid} \label{StretchReverse}
\textbf{Proof of \cref{stretch}$({\Leftarrow})$.}
We need to construct a hamiltonian path~$\hampath$ from $(p,q)$ to $(x,y)$. (For convenience, we will call a diagonal of~$\Bmn$ \emph{inner} if it would be inner with respect to such a hamiltonian path.)
From \pref{stretch-bound}, we know that $e$ is no more than the number of rowful inner subdiagonals, so we may choose a set~$\mathcal{E}$ of $\lceil e/2 \rceil$ rowful inner diagonals. Furthermore:
	\begin{enumerate}
	\item If $e$~is even, we choose each diagonal in~$\mathcal{E}$ to be the union of two distinct subdiagonals.
	\item If $e$~is odd, then \pref{stretch-even} tells us that $m + n$ is also odd, so we may choose $\mathcal{E}$ to contain the diagonal $S_{(m + n - 3)/2}$ that consists of only one subdiagonal.
	\end{enumerate}
Then the diagonals in~$\mathcal{E}$ constitute precisely~$e$ subdiagonals.

Now, applying \cref{IgnoreEast}$({\Leftarrow})$ (repeatedly) to the $\lceil e/2 \rceil$ rowful inner diagonals in~$\mathcal{E}$ and to all $o$~rowful outer diagonals of~$\Bmn$ yields a hamiltonian path~$\hampath$ from $(p,q)$ to~$(x,y)$ in~$\Bmn$. 

Note that all rowful inner diagonals travel north in~$\hampath'$. (Namely, either all non-terminal diagonals travel north, or $m' = n$, in which case, there are no rowful inner diagonals, so the claim is vacuously true.) Therefore, the diagonals in~$\mathcal{E}$ are the only rowful inner diagonals that travel east in~$\hampath$. So exactly~$e$ rowful inner subdiagonals travel east, as desired.
\end{aid}

\begin{aid} \label{TauE}
Write $\term_+ = (x,y)$. Since, by definition, $\term_+$ is the southeasternmost square on~$S_b$, we know that $(x + 1, y - 1)$ is not a square of~$\Bmn$. Therefore, either $x + 1 > m-1$ or $y - 1 < 0$. So either $x = m - 1$ or $y = 0$.
	\begin{itemize}
	
	\item If $x = m - 1$, then $\term_+ E$ is of the form $(0,q)$. The terminal square cannot be $(m-1,n-1)$ \csee{00NotInit}, so $y \neq n - 1$. Therefore $q \neq 0$.
	
	\item If $y = 0$ and $x \neq m - 1$, then $\term_+ E = (x,0)E = (x + 1, 0)$ is of the form $(p,0)$. Also, since $\term_+ \in S_b$, we have $x = b \ge \lceil (m + n - 3)/2 \rceil$, so $p = x + 1 \ge \lceil (m + n - 1)/2 \rceil$. (Furthermore, we have $m \neq n$, for otherwise $\lceil (m + n - 1)/2 \rceil = \lceil (2m - 1)/2 \rceil = m$, so it is impossible to have both $p \le m - 1$ and $p \ge \lceil (m + n - 1)/2 \rceil$.)
	
	\end{itemize}
\end{aid}

\begin{aid} \label{TermFor0qReverse}
\textbf{Proof of \cref{TermFor0q}$({\Leftarrow})$}
We assume the notation of \cref{stretch} (with $p = 0$, because the initial square is $(0,q)$).

\pref{TermFor0q-a} Since $q \le n - 1$ and $x + y = q - 1$, we have $a = q - 1 = x + y$, $o = 0$, and $e_1 = e_2 = 0$. 
Also, since 
	$$ \text{$\lfloor (m - 0 - 1)/2 \rfloor = \mm \ge x$
	\ and \ 
	$\lfloor (m - (m-n) - 1)/2 \rfloor = \nm \le x$}
	,$$
there exists $e \in \{0,1,\ldots,m-n\}$, such that $\lfloor (m - e - 1)/2 \rfloor = x$. 
(That is, $\lfloor (m' - 1)/2 \rfloor = x$.)
Furthermore, we may assume $e$~is even if $m + n$~is even (because the two extremes $0$ and~$m - n$ are even in this case). Then \cref{UsualEndpt} provides a hamiltonian path~$\hampath'$ in~$\board{m'}{n}$ from $(0,q)$ to $(x,y)$, such that all non-terminal diagonals travel north. So \cref{stretch} yields a hamiltonian path from $(0,q)$ to $(x,y)$ in~$\Bmn$.

\pref{TermFor0q-b} 
We have $x + y = b$, $o = 0$, $e_1 = 0$, and $e_2 = e$. 
Also, since 
	$$ \text{$\lfloor (m + 0)/2 \rfloor = \mf \le x$
	\ and \ 
	$\lfloor (m + (m-n))/2 \rfloor = m - \np \ge x$}
	,$$
there exists $e \in \{0,1,\ldots,m-n\}$, such that $\lfloor (m + e)/2 \rfloor = x$. Furthermore, we may assume $e$~is even if $m + n$~is even. Then \cref{UsualEndpt} provides a hamiltonian path~$\hampath'$ in~$\board{m'}{n}$ from $(0,q)$ to $\bigl( \lfloor (m - e)/2 \rfloor, y \bigr) = (x - e,y)$, such that all non-terminal diagonals travel north. So \cref{stretch} yields a hamiltonian path from $(0,q)$ to $(x,y)$ in~$\Bmn$.
\end{aid}

\begin{aid} \label{LargestE}
As was already mentioned in the proof of \fullcref{stretch}{bound}, the number of rowful inner subdiagonals is
	$\max \bigl( \min(m - n, b - a - 1), 0 \bigr)$.
We have
	\begin{align*}
	 b - a - 1 
	 &= (m + n - q - 2) - (q - 1) - 1 
	 = m + n - 2q - 1 
	 \\&\ge m + n - 2(n-1) - 1 
	 = m - n + 1 > m - n 
	 , \end{align*}
so the number of rowful inner subdiagonals is $m - n$. Therefore, the largest possible value of~$e$ is $m - n$. (Note that the requirement in \fullcref{stretch}{even} does not cause a problem, because $m - n$ is even if $m + n$ is even.)
\end{aid}

\begin{aid} \label{TermForp0SmallReverse}
\textbf{Proof of \cref{TermForp0Small}$({\Leftarrow})$}
We assume the notation of \cref{stretch} (with $q = 0$, because the initial square is $(p,0)$). Note that $a = p - 1$, so $e_1 = 0$.

\pref{TermForp0Small-y=nm}
Assume, for the moment, that $p \le n - 1$, so $o = e_2 = 0$.
\Cref{UsualEndpt} provides a hamiltonian path~$\hampath'$ from $(0,p)$ to $(\nm, p - 1 - \nm) = (y,x)$ in~$\board{n}{n}$. The transpose $(\hampath')^*$ is a hamiltonian path from $(p,0)$ to $(x,y)$. Letting $e = m - n$, \cref{stretch} yields a hamiltonian path from $(p,0)$ to $(x,y)$ in~$\Bmn$.

We may now assume that $p \ge n$. We have $a = p - 1$, $o = p - n$, and $e_1 = 0$.
\Cref{init=tauplus} provides a hamiltonian path~$\hampath'$ from $(n,0)$ to $(\nf, \nm)$ in~$\board{n+2}{n}$. Letting $e = m - n$, \cref{stretch} yields a hamiltonian path from $(n + o,0) = (p,0)$ to $(\nf + o, \nm) = (x,y)$ in~$\Bmn$.

\pref{TermForp0Small-y=nf}
We have $e_2 = e$ and also $o = 0$ (since $p \le n - 1$). 
\Cref{UsualEndpt} provides a hamiltonian path~$\hampath'$ from $(0,p)$ to $(\nf, 2n - p -  2 - \nf) = (y, 2n - p -  2 - y)$ in~$\board{n}{n}$. The transpose $(\hampath')^*$ is a hamiltonian path from $(p,0)$ to $(2n - p -  2 - y,y)$. Letting $e = m - n$, \cref{stretch} yields a hamiltonian path from $(p,0)$ to $(x,y)$.
\end{aid}

\begin{aid} \label{TermForp0SmallPf-nm}
We have
	$$ p = m + n - 2 - x - y \ge m + n - 2 - (m-1) - \nf = n - 1 - \nf = \nm .$$
\end{aid}

\begin{aid} \label{TermForp0SmallPf-m'}
Since $a' + b' = m' + n - 3$, we have
	$$ m' = a' + b' + 3 - n = (n-1) + n + 3 - n = n + 2 .$$
\end{aid}

\begin{aid} \label{TermForp0LargeReverse}
\textbf{Proof of \cref{TermForp0Large}$({\Leftarrow})$}
We assume the notation of \cref{stretch} (with $q = 0$, because the initial square is $(p,0)$). Note that $b = p - 1 \ge \lceil (m + n - 3)/2 \rceil$, so $e_1 = e$ and $o = a - n + 1$. 
Also, we have $b - a - 1 = 2p - m -n$.

\pref{TermForp0Large-a}
We have $x + y = a$, so $e_2 = 0$.
Since
	$$ \lfloor (m - 0 - 1)/2 \rfloor = \mm \ge x$$
and
	$$\lfloor (m - (2p-m-n) - 1)/2 \rfloor = m - p + \nm \le x ,$$
there is some $e$, such that $0 \le e \le 2p - m -n$ and $\lfloor (m - e - 1)/2 \rfloor = x$.  Furthermore, we may assume $e$~is even if $m + n$~is even. 
\cref{UsualEndpt} provides a hamiltonian path~$\hampath'$ in~$\board{m'}{n}$ from $(p - o - e,0)$ to $( \lfloor (m' - 1)/2 \rfloor, y' )$, for some~$y'$, such that the terminal square is on the lower subdiagonal of the terminal diagonal. So \cref{stretch} yields a hamiltonian path from $(p,0)$ to $(\lfloor (m' - 1)/2 \rfloor + o,y') = (x,y')$ in~$\Bmn$. Since $(x,y')$ is on the lower subdiagonal~$S_a$ of the terminal diagonal, we must have $y' = y$.
 
\pref{TermForp0Large-b}
We have $x + y = b$, so $e_2 = e$. Since
	$$ \lfloor (m + 0)/2 \rfloor = \mf \le x$$
and 
	$$\lfloor \bigl( m + (b - a - 1) \bigr)/2 \rfloor 
	= \lfloor \bigl( m + (2p - m - n) \bigr)/2 \rfloor
	= p - \np
	\ge x, $$
there is some $e$, such that $0 \le e \le b - a + 1$ and $\lfloor (m + e)/2 \rfloor = x$.  Furthermore, we may assume $e$~is even if $m + n$~is even. 
 \Cref{UsualEndpt} provides a hamiltonian path~$\hampath'$ in~$\board{m'}{n}$ from $(p - o - e,0)$ to $( \lfloor m'/2 \rfloor, y' )$, for some~$y'$, such that the terminal square is on the upper subdiagonal of the terminal diagonal. So \cref{stretch} yields a hamiltonian path from $(p,0)$ to $(\lfloor (m' - 1)/2 \rfloor + o,y')$ in~$\Bmn$. Since $(x,y')$ is on the upper subdiagonal~$S_b$ of the terminal diagonal, we must have $y' = y$.

\pref{TermForp0Large-special}
Since $b = p - 1 = (m + n - 3)/2$, we have $a = b$. Let $e = 0$. 
 \Cref{UsualEndpt} provides a hamiltonian path~$\hampath'$ in~$\board{m'}{n}$ from $(p - o,0)$ to $( \lfloor (m'-1)/2 \rfloor, y' )$, for some~$y'$, such that $( \lfloor (m'-1)/2 \rfloor, y' )$ is on the terminal diagonal~$S_{b-o}$. So \cref{stretch} yields a hamiltonian path from $(p,0)$ to $(\lfloor (m' - 1)/2 \rfloor + o,y') = (\mm, y') = (x,y')$ in~$\Bmn$. Since $(x,y')$ must be on the terminal diagonal~$S_b$, we have $y' = y$.
\end{aid}

\begin{aid} \label{TermForp0LargePf-e}
The largest possible value of~$e$ is
		\begin{align*}
		\max \bigl( \min(m - n , b - a - 1) , 0 \bigr)
		&= \max(  b - a - 1 , 0 )
		\\&= \max \bigl(  (p-1) - (m + n - p - 2) - 1, 0 \bigr)
		\\&= \max ( 2p - m - n , 0 )
		. \end{align*} 
If $2p - m - n < 0$, then, since $p \ge \lceil (m + n - 1)/2 \rceil$, we must have $p = (m + n - 1)/2 = \lceil (m + n - 1)/2 \rceil$. Since  $(m + n - 1)/2 = \lceil (m + n - 1)/2 \rceil$, we know that $m + n$ is odd.
\end{aid}

\begin{aid} \label{TermForp0Small-y=nm-UpperBound}
	If $n$~is even, then $\np = n/2$, so 
		$$ \left\lfloor \frac{m + n}{2} \right\rfloor - 1 - \np = \left\lfloor \frac{m}{2} \right\rfloor - 1 = \mf - 1 .$$
	If $n$~is odd, then $\np = (n+1)/2$, so 
		$$ \left\lfloor \frac{m + n}{2} \right\rfloor - 1 - \np = \left\lfloor \frac{m - 1}{2} \right\rfloor - 1 = \mm - 1 
		= \begin{cases}
		\mf - 1 & \text{if $m$~is odd}, \\
		\mf - 2 & \text{if $m$~is even}
		. \end{cases} $$
\end{aid}

\begin{aid} \label{TermSquaresTopLowerBound}
We have
	\begin{align*}
	\left\lfloor \frac{m + n}{2} \right\rfloor - 1
	&\le \frac{m}{2} + \frac{n}{2} - 1
	= \frac{m}{2} + \frac{n - 1}{2} - \frac{1}{2}
	\\&\le \left( \mf + \frac{1}{2} \right) + \left( \nm + \frac{1}{2} \right) - \frac{1}{2}
	= \mf + \nm + \frac{1}{2} 
	= x + y + \frac{1}{2} 
	. \end{align*}
Since $\lfloor (m + n)/2 \rfloor - 1$ and $x + y$ are integers, this implies 
	$$\lfloor (m + n)/2 \rfloor - 1 \le x + y .$$
\end{aid}

\begin{aid} \label{n=1Pf}
Suppose a hamiltonian path~$\hampath$ in~$\board{m}{1}$ uses the edge $[\sigma](N)$ (and $\sigma N \neq \sigma E$). Then $\hampath$ uses neither $[\sigma](E)$ nor $[\sigma NE^{-1}](E)$. But removing these two directed edges disconnects the digraph, so no hamiltonian path can avoid both of these directed edges. This is a contradiction.
\\[\medskipamount]\centerline{\includegraphics{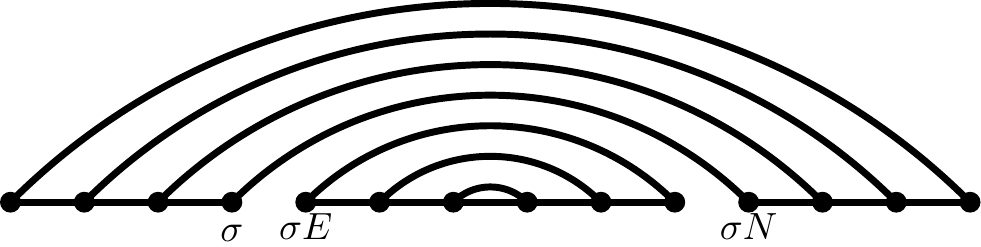}}
%
%
%
%
%
\end{aid}

\begin{aid} \label{n=2Pf}
An exhaustive search (by computer or by hand) tells us that the $13$ hamiltonian paths in \cref{n=2Fig} (on \cpageref{n=2Fig}) are the only hamiltonian paths in~$\board{m}{2}$ for which:
	\begin{itemize}
	\item $m \ge n = 2$,
	\item all non-terminal diagonals travel north,
	and
	\item $\term E \neq \init$, where $\tau$ is the terminal square and $\init$ is the initial square.
	\end{itemize}

\begin{figure}[b]
\hrule\bigskip
\begin{tabular}{ccc}
\raise 1cm\hbox{1.} \includegraphics{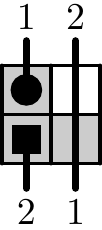} &
\raise 1cm\hbox{2.} \includegraphics{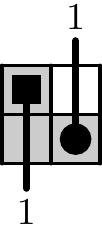} &
\raise 1cm\hbox{3.} \includegraphics{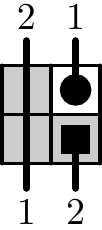} \\
\raise 1cm\hbox{4.} \includegraphics{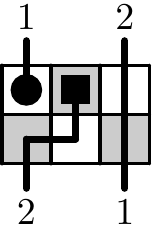} &
\raise 1cm\hbox{5.} \includegraphics{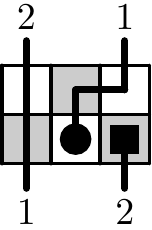} &
\raise 1cm\hbox{6.} \includegraphics{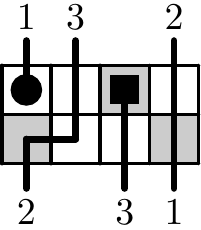} \\
\raise 1cm\hbox{7.} \includegraphics{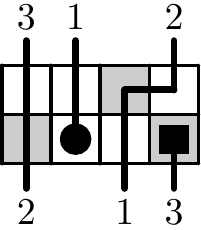} &
\raise 1cm\hbox{8.} \includegraphics{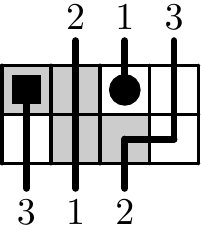} &
\raise 1cm\hbox{9.} \includegraphics{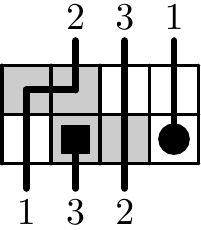} \\
\raise 1cm\hbox{10.} \includegraphics{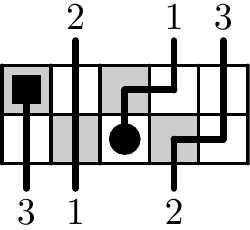} &
\raise 1cm\hbox{11.} \includegraphics{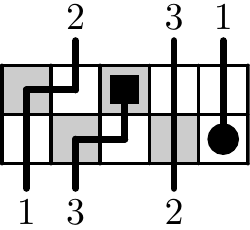} &
\quad\raise 1cm\hbox{12.} \includegraphics{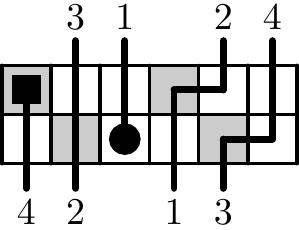}\quad \\
\raise 1cm\hbox{13.} \includegraphics{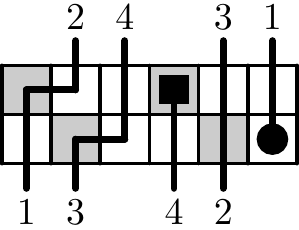} 
\end{tabular}
\caption{Hamiltonian paths in $\board{m}{2}$ in which all non-terminal diagonals travel north and $\term E \neq \init$.}
\label{n=2Fig}
\end{figure}

The only non-rowful diagonals of $\board{m}{2}$ are $S_0 \cup S_{m-1}$ and~$S_m$. Neither of these can be an inner diagonal, so (for $n = 2$) the conclusion of \fullcref{stretch}{HP} can be modified to say that either $(x,y)E = (p,q)$, or $\hampath'$ is one of the hamiltonian paths in \cref{n=2Fig}. For each of the 13 possibilities for~$\hampath'$, here are the initial square and terminal of the corresponding hamiltonian paths in $\board{m}{2}$ that are obtained by applying (the adapted version of) \cref{stretch}:

\begin{enumerate}

\item 
 From $(0,1)$ to $(0,0)$ for $m \ge 2$.

\item 
From $(1,0)$ to $(m-2,1)$ for $m \ge 2$.

\item 
From $(m-1,1)$ to $(m-1,0)$ for $m \ge 2$.

\item \label{HP4} 
From $(0,1)$ to $(m-2,1)$ for odd $m \ge 3$.

\item 
From $(1,0)$ to $(m-1,0)$ for odd $m \ge 3$.

\item \label{HP6} 
From $(0,1)$ to $(m-2,1)$ for $m \ge 4$.

\item 
From $(1,0)$ to $(m-1,0)$ for $m \ge 4$.

\item 
From $(p,1)$ to $(m - 2 - p,1)$ for $m \ge 4$ and  $\mp \le p \le m-2$. 

\item 
From $(p,0)$ to $(m-p,0)$ for $m \ge 4$ and $\mp + 1 \le p \le m - 1$.

\item \label{HP10} 
From $(p,0)$ to $(p-2,1)$ for odd $m \ge 5$ and $2 \le p \le \mm$.

\item \label{HP11} 
From $(p,0)$ to $(p-2,1)$ for odd $m \ge 5$ and $\mf + 2 \le p \le m-1$.

\item \label{HP12} 
From $(p,0)$ to $(p-2,1)$ for $m \ge 6$ and $2 \le p \le \mm$. 

\item \label{HP13} 
From $(p,0)$ to $(p-2,1)$ for $m \ge 6$ and $\mp + 2 \le p \le m-1$. 

\end{enumerate}

These yield the possibilities listed in the table of \fullcref{n=1or2}{2}.  Note:
	\begin{itemize}
	\item \pref{HP4} and \pref{HP6} are combined into row E of the table.
	
	\item \pref{HP10}, \pref{HP11}, \pref{HP12}, and \pref{HP13} are combined into row J of the table.
	\end{itemize}
\end{aid}

\begin{aid} \label{CurranScan}
A PDF scan of \cite{Curran} is available online at 
\\ \url{http://people.uleth.ca/~dave.morris/papers/CurranWitte.pdf}
\end{aid}

\begin{aid} \label{ForbushScan}
A PDF scan of \cite{Forbush} is available online at 
\\ \url{http://people.uleth.ca/~dave.morris/papers/ProjCbds.pdf}
\end{aid}

\end{appendix}

}

\end{document}